\documentclass{daj}

\dajAUTHORdetails{%
    title = {Elekes-Szabó for groups, and approximate subgroups in weak 
    general position},
    runningtitle = {Elekes-Szabó for groups},
    author = {Martin Bays, Jan Dobrowolski, and Tingxiang Zou},
    plaintextauthor = {Martin Bays, Jan Dobrowolski, and Tingxiang Zou},
  keywords = {Elekes-Szab\'o, nilprogressions, approximate subgroups, 
  Mordell-Lang},
}   

\dajEDITORdetails{%
   year={2023},
   number={6},
   received={10 August 2021},   
   published={26 May 2023},  
   doi={10.19086/da.77361},       
}   

\usepackage[utf8]{inputenc}
\usepackage{amssymb,amsmath,enumerate,amsthm,url}
\usepackage[all,cmtip]{xy}
\usepackage{graphicx}
\usepackage{hyperref}
\usepackage{tikz-cd}

\theoremstyle{plain}
\newtheorem{theorem}{Theorem}[section]
\newtheorem{lemma}[theorem]{Lemma}
\newtheorem{proposition}[theorem]{Proposition}

\newtheorem{fact}[theorem]{Fact}
\newtheorem*{fact*}{Fact}
\newtheorem{corollary}[theorem]{Corollary}
\newtheorem{claim}[theorem]{Claim}
\newtheorem*{claim*}{Claim}
\theoremstyle{definition}
\newtheorem{definition}[theorem]{Definition}
\newtheorem*{definition*}{Definition}

\newtheorem*{notation*}{Notation}
\theoremstyle{remark}
\newtheorem{remark}[theorem]{Remark}
\newtheorem*{remark*}{Remark}

\newtheorem*{example*}{Example}

\newtheorem*{note*}{Note}
\newtheorem{question}[theorem]{Question}
\newtheorem*{question*}{Question}

\begin{document}
\newcommand{\G}{\mathbb{G}}
\newcommand{\K}{\mathbb{K}}
\newcommand{\N}{\mathbb{N}}
\newcommand{\C}{\mathbb{C}}
\newcommand{\R}{\mathbb{R}}
\newcommand{\U}{\mathcal{U}}
\newcommand{\GL}{\operatorname{GL}}
\newcommand{\Upp}{\operatorname{Upp}}
\newcommand{\ord}{\operatorname{ord}}
\providecommand{\st}{\operatorname{st}}
\providecommand{\tp}{\operatorname{tp}}
\providecommand{\rad}{\operatorname{rad}}
\providecommand{\eps}{\epsilon}

\renewcommand{\P}{\mathbb{P}}

\providecommand{\defn}[1]{{\bf #1}}
\providecommand{\bdl}{\boldsymbol\delta}
\providecommand{\floor}[1]{\lfloor #1 \rfloor}

\def\Ind#1#2{#1\setbox0=\hbox{$#1x$}\kern\wd0\hbox to 0pt{\hss$#1\mid$\hss}
\lower.9\ht0\hbox to 0pt{\hss$#1\smile$\hss}\kern\wd0}
\def\ind{\mathop{\mathpalette\Ind{}}}
\def\notind#1#2{#1\setbox0=\hbox{$#1x$}\kern\wd0\hbox to 0pt{\mathchardef
\nn=12854\hss$#1\nn$\kern1.4\wd0\hss}\hbox to
0pt{\hss$#1\mid$\hss}\lower.9\ht0 \hbox to
0pt{\hss$#1\smile$\hss}\kern\wd0}
\def\nind{\mathop{\mathpalette\notind{}}}

\begin{frontmatter}[classification=text]

\title{Elekes-Szabó for Groups, and Approximate Subgroups in Weak General 
Position}

\author[mb]{Martin Bays}
\author[jd]{Jan Dobrowolski\thanks{Supported by DFG project BA 6785/2-1.}}
\author[tz]{Tingxiang Zou\thanks{All authors were partially supported by DFG EXC 2044–390685587 and ANR-DFG AAPG2019 (Geomod).}}

\begin{abstract}
  We show that with a suitable weak notion of general position, the 
  Elekes-Szabó condition on the group operation of a connected complex 
  algebraic group characterises nilpotence of the group.
  Along the way, we prove a Mordell-Lang result for generic finitely generated 
  subgroups of commutative complex algebraic groups.
\end{abstract}
\end{frontmatter}

\section{Introduction}
The main theorem of \cite{ES-groups} shows that any complex subvariety $V 
\subseteq (\C^n)^3$ which has asymptotically large intersection with products 
of finite subsets of $\C^n$ in \emph{general position} must be, up to finite 
correspondences, the graph of an algebraic group operation. 
\cite{Breuillard-Wang} showed that the group involved must be abelian, and it 
follows that the result has content even when $V$ is taken to be the graph of 
an algebraic group operation: it shows that if $V$ has such large 
intersections with products of sets in general position, then the group must 
be abelian-by-finite. Here, general position of a subset $X \subseteq \C^n$ 
means that $X$ has boundedly finite intersections with proper subvarieties $W 
\subsetneq \C^n$. This assumption was slightly weakened in \cite{BB-cohMod} to 
\emph{coarse general position}, which allows asymptotic growth in such 
intersections slower than $|X|^\epsilon$ for all $\epsilon\in \R_{>0}$, while 
retaining the same conclusions, including commutativity of the resulting 
group. Conversely, it was shown there that the graph of every connected 
commutative complex algebraic group admits asymptotically large intersections 
with products of finite sets in coarse general position.
shown there that with this coarse general position assumption, every connected 
commutative complex algebraic group arises in this way.

Here we make a start on understanding what happens to these Elekes-Szabó 
phenomena when the general position assumption is relaxed yet further, by 
considering this special case that $V$ is the graph of a complex algebraic 
group operation. We show that, just as coarse general position in the 
Elekes-Szabó configuration characterises when the group is abelian-by-finite, 
a suitably weakened general position assumption characterises when it is 
nilpotent-by-finite.
Specifically, we prove Theorem~\ref{t:main-comb}.

\begin{definition} \label{defn:tauwgp}
  Let $W$ be a complex variety.
  Let $\alpha,\tau \in \N$.
  A finite subset $X \subset W$ is in \defn{weak $(\alpha,\tau)$-general position} 
  in $W$ if $|X| \geq \alpha$ and for any proper irreducible complex closed 
  subvariety $W' \subsetneq W$ of complexity at most $\alpha$,
  we have $|W' \cap X| \leq |X|^{1-\frac 1\tau}$.
\end{definition}

\begin{theorem}\label{t:main-comb}
  Let $G$ be a complex connected algebraic group.
  Then the following are equivalent.
  \begin{enumerate}[(I)]
    \item $G$ is nilpotent;
    \item The graph $\Gamma_G$ of the group operation is \emph{wgp-coherent}:
      For all $\epsilon > 0$ and $\alpha \in \N$
      there exist $\tau\in\N$
      and finite subsets $A,B,C \subseteq G$
      which are in weak $(\alpha,\tau)$-general position in $G$
      such that 
       $|\{(a,b,ab) \in A\times B\times C\}| \geq \max(|A|,|B|,|C|)^{2-\epsilon}$.
  \end{enumerate}
\end{theorem}

As mentioned above, the analogous statement with coarse general position and 
commutativity was observed in \cite[Corollary~3.15]{BB-cohMod}, and the $(II) 
\Rightarrow (I)$ direction of this abelian case was previously obtained in 
\cite{Breuillard-Wang}. Our proof of this direction of 
Theorem~\ref{t:main-comb} is along similar lines, and was heavily inspired by 
the ideas of \cite{Breuillard-Wang}. Briefly, the idea in both cases is to 
apply the Balog-Szemeredi-Gowers-Tao theorem to reduce to considering the 
existence of suitably approximate subgroups in suitably general position, and 
then apply a variation on \cite{BGT-lin} to obtain nilpotence or commutativity 
of the group. This is treated in Section~\ref{s:wgp}.
We also observe in Proposition~\ref{p:sLPgpAb} that the commutative case goes 
through when coarse general position is weakened to a certain 
condition related to a result of Larsen and Pink, intermediate between coarse 
and weak general position.

For $(I) \Rightarrow (II)$, we will see that we can take $A=B=C$ to be the 
approximate subgroup given by a nilprogression generated by an $l$-tuple of 
independent generics of $G$, for a sufficiently large $l$ depending only on 
$G$. The difficulty is to see that such nilprogressions have small 
intersection with proper subvarieties.

We first prove this in Section~\ref{s:vectorialML} in the case that $G$ is 
commutative (where a nilprogression is just a generalised arithmetic 
progression). For $G$ a semiabelian variety, it is a consequence of the truth 
of the Mordell-Lang conjecture that this holds with $l=1$. For $G = \G_a^n$ 
(where $\G_a$ denotes the additive group over $\C$), it is straightforward to 
see that it holds with $l=n$. A general commutative algebraic group in 
characteristic 0 can be seen as an extension of a semiabelian variety by some 
$\G_a^n$. Although the general case does not just follow from the two extreme 
cases, we obtain it by adapting techniques from Hrushovski's proof  
\cite{Hr-ML} of Mordell-Lang for function fields.

In Section~\ref{s:nilproWGP}, we extend this inductively to nilpotent 
algebraic groups. Let us sketch the method in the case of a group $G$ of 
nilpotency class 2.
If we take independent generic generators $g_1,...,g_{2n}$, the intersection 
of the nilprogression $P$ with the derived subgroup $G'$ will contain the 
arithmetic progression generated by the commutators 
$[g_1,g_2],...,[g_{2n-1},g_{2n}]$. It may be that these commutators are not 
themselves generic in $G'$, but if $w$ is the commutator width of $G$ then 
products of $w$ of them will be independent generics. Moreover, the fibres of 
the quotient $P/G'$ will essentially be unions of translates of such 
arithmetic progressions. Replacing the nilprogression with a ``nilbox'' as in 
\cite{BG09} and applying further transformations, we will reduce to the case 
that the fibres are literally of this form. Meanwhile, $P/G'$ is itself an 
arithmetic progression in independent generators $g_i/G'$ in the commutative 
group $G/G'$. Applying the commutative case of Section~\ref{s:vectorialML} to 
$G'$ and to $G/G'$ will allow us to conclude.

In fact, we work throughout with pseudofinite sets in the style of 
\cite{Hr-psfDims}, as explained in Section~\ref{s:setup}. This allows us to 
deal more explicitly with the structure behind the combinatorics.

\subsection{Setup}\label{s:setup}
Let $\U$ be a non-principal ultrafilter on $\N$.
For a set $S$ the \defn{internal} subsets of the ultrapower $S^\U$ are
the elements of $\P(S)^\U \subseteq \P(S^\U)$.
The finite cardinality function $\P(S) \to \N \cup \{\infty\}$ induces the non-standard finite cardinality function, which assigns to every internal subset $X=\lim_{i \to \U} X_i\subseteq S^\U$ the value $|X|:= \lim_{i\to \U} |X_i|\in (\N\cup \{\infty\})^\U=\N^\U\cup\{\infty\}$.
We fix some $\xi \in \R^\U$ greater than any standard real,
and for $\alpha \in \R^\U_{\geq 0}$ we set $\bdl(\alpha) = \bdl_\xi(\alpha) := 
\st \log_\xi \alpha \in \R \cup \{-\infty,\infty\}$,
where $\st : \R^\U \cup \{-\infty,\infty\} \to \R \cup \{-\infty,\infty\}$ is 
the standard part map, and $\log_{\xi} \alpha=\lim_{i\to \U}\log_{\xi_i}\alpha_i$ for $\xi=\lim_{i\to \U}\xi_i$ and $\alpha=\lim_{i\to \U} \alpha_i$.
We set $\bdl(X) := \bdl(|X|)$.

A \defn{$\bigwedge$-internal} subset of $S^\U$ is the intersection $\bigcap_{i 
\in \omega} X_i$ of
countably many internal subsets $X_i$. We define $\bdl(\bigcap_{i \in \omega} 
X_i) := \inf_{n \in \omega} \bdl(\bigcap_{i<n} X_i)$.

Following Wagner \cite{W-dimGroupField}, we say a $\bigwedge$-internal set $X$ 
is \defn{broad} if $0 < \bdl(X) < \infty$.

If $S_i$ are arbitrary sets, their ultraproduct $\prod_{i \to \U} S_i$ is an
internal subset of $(\bigcup_{i \in \N} S_i)^\U$, so the above definitions
also apply to such ultraproducts.

We will use the $\aleph_1$-compactness of ultraproducts over $\U$, which means 
that a countable set of internal sets has non-empty intersection if any finite 
subset has. This applies not only to internal subsets of $S^\U$, but also to 
internal subsets of the set $\P(S)^\U$ of internal subsets of $S^\U$, and so 
on.
This implies in particular that any $\bigwedge$-internal subset $X = 
\bigwedge_i X_i$ of $S^\U$ contains an internal subset of the same $\bdl$.
Indeed, suppose $a := \bdl(X) \in \R$ (the case $a=\infty$ is similar).
Then finite intersections of
\[\{ \{ Z \in \P(S)^\U : |Z| \geq \xi^{a - \frac1m} \} : m \}
\cup \{ \{ Z \in \P(S)^\U : Z \subseteq X_i \} : i \}\]
are non-empty, as witnessed by $\bigcap_{i<n} X_i$, and we conclude by 
$\aleph_1$-compactness.

Let $\K := \C^\U$.

Unless otherwise specified, all algebraic varieties will be over $\K$, and we 
identify them with their sets of $\K$-rational points.

\section{Approximate subgroups in weak general position}\label{s:wgp}

In this section we prove the $(II) \Rightarrow (I)$ direction of 
Theorem~\ref{t:main-comb}, which we reformulate in pseudofinite terms and in 
more detail as Theorem~\ref{t:BGT-wgp}. The core of the argument is 
essentially a matter of \cite[Theorem~2.5]{BGT-lin}, but the situation is a 
little different, and we take the opportunity to obtain a suitable analogue of 
it in terms of $\bigwedge$-internal subgroups (Proposition~\ref{p:bgt}) as a 
consequence of \cite[Theorem~3.1]{Hr-psfDims} (which is itself a recasting of 
the arguments in \cite{BGT-lin}) and some further arguments taken from 
\cite{BGT-lin} and \cite{BG}. The relation to \cite{BGT-lin} is addressed 
further in Remark~\ref{r:bgtComp}. The other main ingredient is the 
Balog-Szemeredi-Gowers-Tao theorem \cite[5.4]{Tao-BSG}, which is behind 
Lemma~\ref{l:BSGT} below.

\subsection{Broad $\bigwedge$-definable subfields and subgroups}
We first recall some known consequences of the fact that $\K$ is an 
ultraproduct of characteristic 0 fields.

First, a version of the sum-product theorem in this context.
Write $\G_a$ and $\G_m$ for the additive and multiplicative algebraic groups 
(over $\K$).

\begin{fact} \label{f:sumProd}
  Suppose $A \leq  \G_a$ is a $\bigwedge$-internal subgroup with $\bdl(A) < 
  \infty$
  and $M \leq  \G_m$ is a broad $\bigwedge$-internal subgroup
  and $M\cdot A \subseteq  A$.
  Then $A$ is trivial.
\end{fact}
\begin{proof}
  We give the special case for this setting of 
  \cite[Corollary~3.10]{W-dimGroupField}.

  Let $M' := M - M$ and let $R_n := \sum_{i=1}^n M'$, so $R := \bigcup_{n < 
  \omega} R_n$ is a subring of $\K$.
  Suppose $A$ is non-trivial, say $a \in A \setminus \{0\}$.
  Then $r \mapsto  ra$ is an internal injection $R \rightarrow  A$,
  so $\bdl(R) := \sup_n \bdl(R_n) \leq  \bdl(A) < \infty$, so say $m \in 
  \omega$ is such that $2\bdl(R_m) > \bdl(R)$. Let $X := R_m$.
  Then for $r,s \in R \setminus \{0\}$, the map $X^2 \rightarrow  R$ given by 
  $(x,y) \mapsto  rx+sy$ is not injective, so we have some equality $rx+sy = 
  rx'+sy'$ with say $y\neq y'$, so $\frac sr = \frac{x-x'}{y'-y}$. So we find 
  that $\frac{X-X}{X-X}$ is the fraction field of $R$. So we have a broad 
  $\bigwedge$-internal subfield of $\K$, which by sum-product 
  (\cite[Proposition~3.1]{W-dimGroupField} or \cite[Example~3.5]{Hr-psfDims}) 
  is internal. But $\K = \C^\U$, and $\C$ contains no finite subfield.
\end{proof}

Next, a version of \cite[Proposition~6.2]{BGT-lin}.
\begin{fact}\label{simple_case}
  A simple linear algebraic group $G$ over $\K$ has no $\bigwedge$-internal 
  broad Zariski-dense subgroup.
\end{fact}
\begin{proof}
  If $H$ is such a subgroup, then by \cite[Theorem 3.1]{Hr-psfDims} there is a 
  broad internal dense subgroup $H_1\leq G$. By Jordan's theorem 
  \cite[Theorem~0.1]{LP}, $H_1$ has an abelian subgroup of finite index, and 
  then by density so does $G$.
  But $G$ is infinite since $H$ is broad, so this contradicts simplicity.
\end{proof}

\subsection{Broad subgroups in homomorphic weak position}
\begin{definition}
  If $G$ is an algebraic group over $\K$,
  and $X \subseteq G(\K)$ is a broad $\bigwedge$-internal subset,
  say $X$ is in \defn{homomorphic weak general position} (\defn{hwgp}) in $G$
  if $X/N$ is broad for any proper normal algebraic subgroup $N \triangleleft 
  G$.
\end{definition}
\begin{remark}
  If $X \subseteq G(\K)$ is hwgp in $G$ and $\theta : G \twoheadrightarrow H$ 
  is a surjective homomorphism of algebraic groups, then $\theta(X)$ is hwgp 
  in $H(\K)$.
\end{remark}

Let $G$ be an algebraic group. The \emph{radical} $\rad(G)$ of $G$ is the 
largest solvable connected algebraic subgroup of $G$.
$G$ is \emph{semisimple} if $\rad(G)$ is trivial, and \emph{almost simple} 
if $G/Z(G)$ is simple.

\begin{fact}[{\cite[21.51]{Milne-algGrps}}]\label{decomposition}
  Let $G$ be a connected semisimple linear algebraic group. Then there are 
  almost simple algebraic groups $G_1,\dots,G_k$ and an epimorphism $\pi: 
  G_1\times\dots\times G_k\to G$ with finite kernel.
\end{fact}

\begin{proposition} \label{p:bgt}
  Let $G$ be a connected algebraic group over $\K$.
  Let $H \leq G(\K)$ be a broad hwgp Zariski-dense $\bigwedge$-internal 
  subgroup.
  Then $G$ is nilpotent.
\end{proposition}
\begin{proof}
  First, we reduce to the case that $G$ is linear.
  By e.g.\ \cite[Theorem~13]{rosenlicht-algGroups}, $G/Z(G)$ is linear.
  If $Z(G)=G$, we are done.
  Else, $H/Z(G)$ is a broad hwgp Zariski-dense $\bigwedge$-internal subgroup of 
  $G/Z(G)$, and if $G/Z(G)$ is nilpotent then so is $G$.
  So we may assume that $G$ is linear, say $G \leq \GL_n$.

  \underline{Step I}: $G$ is solvable.

  Suppose not; then $\rad(G)$ is a proper normal algebraic subgroup of $G$, so 
  by hwgp $H/\rad(G)$ is broad. By Fact \ref{decomposition}, there are almost 
  simple algebraic groups $G_1,\dots,G_k$ and an epimorphism $\pi: 
  G_1\times\dots\times G_k \to G/\rad(G)$ with finite kernel.

  Put $H_1:=\pi^{-1}(H/\rad(G))$. As $\ker(\pi)$ is finite and $H$ is broad, 
  $H_1$ is broad as well. Now $H_1$ embeds into $\pi_1(H_1)\times\dots\times 
  \pi_k(H_1)$ (where $\pi_i:G_1\times\dots\times G_k\to G_i$ are the 
  projections) via $h\mapsto (\pi_1(h),\dots,\pi_k(h))$. Thus, by additivity 
  of $\bdl$ we have that $\pi_i(H_1)$ is broad for some $i$.
  Now $Z(G_i)$ is finite, since otherwise
  $\pi(0\times0 \times ... \times Z(G_i)\times ... \times 0 \times 0)$ is an
  infinite abelian normal subgroup of $G/\rad(G)$, contradicting triviality of
  $\rad(G/\rad(G))$.
  So $\pi_i(H_1)/Z(G_i)$ is also broad.

  Now $H_1$ is Zariski-dense in $\prod_i G_i$. Indeed,
  if $U \subseteq \prod_i G_i$ is Zariski-open and disjoint from $H_1$,
  then $\pi(U) \subseteq G/\rad(G)$ is disjoint from $H/\rad(G)$,
  but $\dim(\pi(U)) = \dim(U) = \dim(\prod_i G_i) = \dim(G/\rad(G))$,
  so this contradicts Zariski-density of $H/\rad(G)$.

  Hence $\pi_i(H_1)/Z(G_i)$ is Zariski-dense in $G_i/Z(G_i)$.
  As $G_i/Z(G_i)$ is simple and linear, this contradicts Fact 
  \ref{simple_case}.

  \underline{Step II}: $G$ is nilpotent.

  We proceed by translating the proof of \cite{BG}.
  By Maltsev, by conjugating by an element of $\GL_n$, we may assume that $G$ 
  is a subgroup of $\Upp_n$.

  We argue by induction on $n$.

  Let $\pi_1 : \Upp_n \rightarrow  \Upp_{n-1}$ be the homomorphism defined by 
  deleting the first row and first column,
  and $\pi_2 : \Upp_n \rightarrow  \Upp_{n-1}$ the homomorphism defined by 
  deleting the last row and last column.

  Let $h_1,\ldots ,h_n \in H$.
  Let $h := [h_1,h']$ where
  $h' := [h_2,\ldots ,[h_{n-1},h_n]\ldots ]$.
  We show $h=1$.

  Now for $i=1,2$ we argue that $\pi_i(G)$ is nilpotent of step $n-1$:
  if $\pi_i(G)=0$, this is clear, and otherwise since $H$ is hwgp, 
  $\pi_i(H) \leq  \pi_i(G)$ is a broad hwgp Zariski-dense 
  $\bigwedge$-internal subgroup, and we apply the inductive hypothesis.

  Hence both $h$ and $h'$ are in $N := \ker\pi_1 \cap \ker\pi_2 =
  \begin{pmatrix}
    1 & 0 & \ldots & 0 & * \\
    0 & 1 & \ldots & 0 & 0 \\
    \vdots & \vdots & \vdots & \vdots & \vdots \\
    0 & 0 & \ldots & 0 & 1
  \end{pmatrix}
  $.
  The conjugation action of $X := (x_{ij})_{ij}\in G$ on $N$ is by multiplication 
  in the top-right corner by $\theta(X) := \frac{x_{11}}{x_{nn}}$.
  Note that $\theta : G \rightarrow  \G_m$ is a homomorphism.
  Let $\alpha : N \cong \G_a$ be the isomorphism $\alpha((x_{ij})_{ij}) := 
  x_{1n}$.

  Suppose first that $\ker\theta$ is a proper subgroup of $G$.
  Then since $H$ is hwgp, $\theta(H)$ is a broad $\bigwedge$-internal 
  subgroup of $\G_m(\K)$.
  But $\theta(H)$ acts by multiplication on $\alpha(H \cap N)$ via 
  conjugation, so by Fact~\ref{f:sumProd}, $H\cap N$ is trivial, so $h=1$.

  It remains to consider the case $\ker\theta = G$.
  Then $G$ commutes with $N$, so in particular $h = [h_1,h'] = 1$.

  So we conclude that $H$ is nilpotent of step $n$, and then by Zariski 
  density so is $G$.
\end{proof}

\subsection{Coarse-approximate subgroups}
\begin{definition}
  Let $(G;\cdot)$ be an internal group (meaning that $G$ is an internal subset
  of some $S^\U$ and the graph of the group operation $\Gamma_G \subseteq G^3$ 
  is also internal), and let
  $X\subseteq G$ be internal. We call $X$ a \defn{coarsely approximate subgroup}  
  if $X=X^{-1}$, $\mathsf{1}\in X$, $X$ is broad, and $XX\subseteq KX$ where
  $K\subseteq G$ is internal with $\bdl(K) = 0$.
\end{definition}

\begin{proposition} \label{p:wedgify}
  If $X \subseteq G$ is a coarsely approximate subgroup, then there exists a
  $\bigwedge$-internal subgroup $H \supseteq X$ with $\bdl(H) = \bdl(X)$.
  Conversely, if $H \leq G$ is a broad $\bigwedge$-internal subgroup,
  then there exists a coarsely approximate subgroup $X \subseteq H$ with
  $\bdl(X) = \bdl(H)$.
\end{proposition}
\begin{proof}
  Suppose $X$ is a coarsely approximate subgroup and $XX\subseteq KX$, where $\bdl(K)=0$.
  We can assume $\bdl(X) = 1$.
  Let $\nu := \floor{\log_2(\log_{|K|}\xi)} \in \N^\U$.
  Then $\nu > \N$ since $\bdl(K) = 0$, but $\bdl(|K|^{2^{\nu}}) \leq 1$. 
  For $n \in \N$, let $Y_n := X^{2^{\nu-n}}$.
  (Explicitly, if $X = \prod_{s \to \U} X_s$ and $\nu = \lim_{s \to \U} \nu_s$
  with $\nu_s \in \N$, then $Y_n = \prod_{s \to \U} X_s^{2^{\nu_s - n}}$.)
  Then $Y_nY_n = Y_{n-1}$, and $\bdl(Y_n) \leq \bdl(K^{2^{\nu-n}-1}X) \leq
  2^{-n} + 1$.
  So $H := \bigcap_n Y_n$ is as required.

  Suppose conversely that $H$ is a broad $\bigwedge$-internal subgroup. By $\aleph_1$-compactness (see Section~\ref{s:setup}), there is 
  an internal subset $X \subseteq H$ with $\bdl(X)=\bdl(H)$. We can assume 
  that $X$ is closed under inverse and contains the identity. Now $X^3 
  \subseteq H^3 = H$, so $\bdl(X^3) = \bdl(X)$. It follows by 
  \cite[Corollary~3.11]{Tao-BSG} that $X^3$ is a coarsely approximate subgroup.
\end{proof}

\renewcommand{\gg}{\mathfrak{g}}

%

\subsection{Weak general position}
\begin{definition}
  An internal subset $X$ of an irreducible algebraic variety $V$ is in 
  \defn{weak general position} (for short: $X$ is \defn{wgp}) in $V$ if for 
  any proper subvariety $W \subseteq  V$ (over $\K$), $\bdl(X \cap W) < \bdl(X)$.

  Given $\eta > 0$, we say $X$ is wgp with \defn{gap} $\eta$ if moreover
  $\bdl(X \cap W) \leq  \bdl(X) - \eta$ for every $W$. We say $X$ is \defn{wgp 
  with a gap} if this holds for some $\eta>0$.
\end{definition}
(See the proof of Theorem~\ref{t:main-comb} below for how this relates to the 
definition of weak $(\alpha,\tau)$-general position in the introduction.)

\begin{lemma} \label{l:wgp-hwgp}
  A broad internal wgp subset of an algebraic group is hwgp.
\end{lemma}
\begin{proof}
  Suppose $X \subseteq  G$ is not hwgp, so say $\bdl(X / H) = 0$ where $H < G$ is a 
  proper normal algebraic subgroup. Now for all $n \in \N$, we have $|X/H| < 
  \xi^{\frac1n}$ and hence, by the pigeonhole principle $ 
  |X \cap \gamma H| \geq  |X|\xi^{-\frac1n}$ for some $\gamma \in G$.
  Applying $\aleph_1$-compactness, we find $\gamma \in G$ with $\bdl(X \cap 
  \gamma H) = \bdl(X)$, contradicting wgp.
\end{proof}

\begin{definition}
  Let $G$ be a connected algebraic group over $\K$.

  Say the graph of the group operation $\Gamma_G$ is \defn{wgp-coherent}
  if there are broad internal subsets $A,B,C \subseteq  G$ such that $\bdl((A\times 
  B\times C) \cap \Gamma_G) = 2\bdl(A) = 2\bdl(B) = 2\bdl(C)$
  and $A,B,C$ are wgp in $G$.
\end{definition}

\begin{remark}
  We could also consider an \emph{a priori} weaker notion in which $A,B,C$ are 
  $\bigwedge$-internal rather than internal. However, an $\aleph_1$-compactness 
  argument shows that this is actually equivalent: take internal subsets of 
  $A,B,C$ with the same $\bdl$, hence retaining wgp, and such that the 
  intersection with $\Gamma_G$ also has the same $\bdl$.
\end{remark}

\begin{lemma} \label{l:BSGT}
  Let $G$ be a connected algebraic group over $\K$.
  Suppose $\Gamma_G$ is wgp-coherent.
  Then there exists an hwgp Zariski-dense coarsely approximate subgroup $X \subseteq  
  G$.
\end{lemma}
\begin{proof}

  Let $A,B,C \subseteq  G$ be as in the definition of wgp-coherent.
  We may assume $\bdl(A)=\bdl(B)=\bdl(C)=1$.

  Let $E := \{ (a,b) \in A \times B : a\cdot b \in C \} \subseteq  A\times B$
  and $P := \{ a\cdot b : (a,b) \in E \} \subseteq  C$.
  Then $\bdl(E) = \bdl((A \times B \times C) \cap \Gamma_G) = 2$,
  and $\bdl(P) \leq  \bdl(C) = 1$.

  We now apply \cite[5.4 $(ii) \Rightarrow  (iv)$]{Tao-BSG}
  to deduce that there is a coarsely approximate subgroup $X$ of $G$ and $x
  \in G$ such that $\bdl(X) = \bdl(A) = \bdl((x\cdot X)\cap A)$.
  Indeed, let $K := \max(\frac{|A\|B|}{|E|},\frac{|P|}{|A|^{1/2}|B|^{1/2}})$.
  Then $\bdl(K) = 0$,
  i.e.\ $K \leq  \xi^\eps$ for all $\eps > 0$. Applying the above cited 
  result in each co-ordinate of the ultraproduct with
  $-1$ as the first constant exponent in ``$K^{O(1)}$'' and $1$ as the second 
  such exponent in 5.4(ii),
  we obtain 5.4(iv) where each constant is $K^{O(1)}$ hence has coarse dimension 
  $O(1)\bdl(K) = 0$.

  Then $(x\cdot X) \cap A \subseteq  A$ is wgp in $G$, being a subset of $A$ with the 
  same $\bdl$.
  So its translate $Y := X \cap (x^{-1}\cdot A) \subseteq  X$ is also wgp.
  In particular $Y$ is Zariski-dense, and by Lemma~\ref{l:wgp-hwgp} it is hwgp.
  Since hwgp and Zariski-density are preserved by taking broad supersets,
  $X$ is also hwgp and Zariski-dense in $G$.
\end{proof}

\subsection{Main theorem}

\begin{theorem} \label{t:BGT-wgp}
  Let $G$ be a connected complex algebraic group.

  The following are equivalent:
  \begin{enumerate}[(i)]\item $\Gamma_G$ is wgp-coherent.
  \item There exists an hwgp Zariski-dense coarsely approximate subgroup of 
  $G$.
  \item There exists a broad hwgp Zariski-dense $\bigwedge$-internal subgroup of $G$.
  \item $G$ is nilpotent.
  \item There exists a coarsely approximate subgroup of $G$ which is wgp with a 
  gap.
  \item[(v')] There exists a coarsely approximate subgroup of $G$ which is wgp.
  \end{enumerate}
\end{theorem}
\begin{proof}
  \underline{(i) $\Rightarrow $ (ii)}:
    Lemma~\ref{l:BSGT}.

  \underline{(ii) $\Rightarrow $ (iii)}:
    Proposition~\ref{p:wedgify}, noting that a superset of hwgp is hwgp.

  \underline{(iii) $\Rightarrow $ (iv)}:
    Proposition~\ref{p:bgt}.

  \underline{(iv) $\Rightarrow $ (v)}:
    Theorem~\ref{wgp}.

  \underline{(v) $\Rightarrow $ (v')}:
    Immediate.

  \underline{(v') $\Rightarrow $ (i)}:
    Take $A=B=C:=X^2$. As $X$ is coarsely approximate, we have $\bdl(X^2)=\bdl(X)$, so
    $$\bdl((X^2\times X^2\times X^2)\cap \Gamma)\geq \bdl((X\times X\times X^2)\cap \Gamma)=\bdl(X\times X)=2\bdl(X)=2\bdl(X^2),$$
    as required.
\end{proof}

\begin{remark}
  With the possible exception of (iv)$\Rightarrow $(v), all parts of this proof go 
  through with $G$ over $\K$ rather than over $\C$. In particular, (i) $\Rightarrow $ 
  (iv) holds for $G$ over $\K$.
\end{remark}

\begin{question}
  Does Theorem~\ref{wgp}, and hence Theorem~\ref{t:BGT-wgp}, also hold for $G$ over 
  $\K$?
\end{question}

\begin{proof}[Proof of Theorem~\ref{t:main-comb}]
  By Theorem~\ref{t:BGT-wgp}, it suffices to see that the condition 
  Theorem~\ref{t:main-comb}(II) is indeed equivalent to wgp-coherence.

  Theorem~\ref{t:main-comb}(II) holds iff there exist infinite pseudofinite subsets
  $A=\prod_{i\to \U} A_i$, $B=\prod_{i\to \U} B_i$, $C=\prod_{i\to \U}C_i
$ of   $G$ with $|A| = |B| = |C|$
  and $|(A\times B\times C) \cap \Gamma_G| \geq  |A|^{2-\epsilon}$ for all 
  $\epsilon>0$
  such that for all $\alpha$ there exists $\tau$ such that $A_i,B_i,C_i$ are 
  $(\alpha,\tau)$-wgp in $G$ for $\U$-many $i$.
  Setting $\xi := |A|$, this is equivalent to the existence of internal 
  $A,B,C$ with $\bdl(A)=\bdl(B)=\bdl(C)=1$
  and $\bdl((A\times B\times C) \cap \Gamma_G) = 2$ and $A,B,C$ wgp in $G$,
  as required.
\end{proof}

\begin{remark}
  If we weaken the weak general position assumption in the definition of 
  coherence to just Zariski-density, the resulting notion is not useful:
  if $G$ is any infinite algebraic group, then there are broad internal 
  Zariski-dense $X_1,X_2,X_3$ with $\bdl(X_i)=1$ and $\bdl(\prod_i X_i \cap 
  \Gamma_G) = 2$.

  Indeed, let $H \leq  G$ be a non-trivial commutative connected algebraic 
  subgroup (for example, the Zariski closure of the subgroup generated by an 
  element of infinite order; such an element exists since we are in 
  characteristic 0. Alternatively, it follows from Reineke's theorem in model 
  theory (\cite{Re75}) that any minimal infinite algebraic subgroup is abelian).
  Then $\Gamma_H \subseteq  H^3$ is wgp-coherent by Theorem~\ref{t:BGT-wgp},
  witnessed say by $X'_1,X'_2,X'_3 \subseteq  H$.
  Let $Y \subseteq  G$ be an internal Zariski dense subset with $\bdl(Y)=0$;
  to see that this exists, note that a large enough finite set of independent 
  generics is contained in no constructible set of bounded complexity, and 
  apply $\aleph_1$-compactness.
  Then $X_i := X'_i \cup Y$ are as required.
\end{remark}

\begin{remark} \label{r:bgtComp}
  Our proof of Theorem~\ref{t:main-comb}(i)$\Rightarrow $(iv) is an adaptation of arguments in 
  \cite{BGT-lin}. An alternative proof can be obtained (for linear groups) 
  from the statement of \cite[Theorem~2.5]{BGT-lin}, as follows.

  Let $G \leq  \operatorname{GL}_n(\K)$ be a connected linear algebraic group over $\K$.
  Suppose $\Gamma_G$ is wgp-coherent.
  Define ``coset weak general position'' (coset-wgp) by analogy with hwgp, but 
  without the normality condition: a broad $\bigwedge$-internal set $X \subseteq  G(\K)$ is 
  coset-wgp in $G$ if the set of left cosets $X/J$ is broad for any proper 
  algebraic subgroup $J \leq  G$.
  Then Lemma~\ref{l:wgp-hwgp} goes through for coset-wgp, so exactly as in 
  Lemma~\ref{l:BSGT} we obtain a coset-wgp coarsely approximate subgroup $H$ of $G$.

  Applying \cite[Theorem~2.5]{BGT-lin} to the co-ordinates of the 
  ultraproduct, we obtain a coarsely approximate subgroup $H'$ of $G$ such that 
  $H \subseteq  XH'$ for some internal $X$ with $\bdl(X) = 0$,
  and $H'$ generates a subgroup of nilpotency class $\leq n-1$;
  this latter condition passes to the ultraproduct, since it is witnessed by 
  the triviality of certain words.
  Let $J$ be the Zariski closure of the subgroup generated by $H'$.
  Then $J$ has nilpotency class $\leq n-1$ (by considering vanishing of words),
  and $H \subseteq  XJ$,
  which contradicts coset-wgp unless $J=G$.
  So $G$ is nilpotent (of nilpotency class $\leq n-1$).
\end{remark}

\subsection{Strict Larsen-Pink general position}
As mentioned in the introduction, a consequence of \cite{BB-cohMod} is that a 
connected complex algebraic group is commutative iff it is cgp-coherent. Here 
we strengthen this a little, by showing that a weaker general position 
assumption is sufficient to deduce commutativity. The argument is inspired by 
the original argument for commutativity in \cite{Breuillard-Wang}, which 
involved first arguing for nilpotence as a consequence of \cite{BGT-lin} and 
then showing that general position is contradicted if the derived subgroup is 
nontrivial. However, our argument does not go via nilpotence.

\begin{definition}
  A broad $\bigwedge$-internal subset $X$ of an algebraic variety $V$ is in 
  \defn{strict Larsen-Pink general position} (\defn{sLPgp}) if $\bdl(X \cap W) 
  < \frac{\dim(W)}{\dim(V)}\bdl(X)$ when $W \subseteq  V$ is a subvariety with $0 < 
  \dim(W) < \dim(V)$.
  
   For a connected algebraic group $G$ over $\K$, say the graph of the group operation $\Gamma_G$ is \defn{sLPgp-coherent}
  if there are broad internal subsets $A,B,C \subseteq  G$ such that $\bdl((A\times 
  B\times C) \cap \Gamma_G) = 2\bdl(A) = 2\bdl(B) = 2\bdl(C)$
  and $A,B,C$ are sLPgp in $G$.
\end{definition}

\begin{remark}
  We can also consider the weaker condition with $\bdl(X \cap W) \leq  
  \frac{\dim(W)}{\dim(V)}\bdl(X)$, which as in \cite[Question~7.19]{BB-cohMod} 
  we call \emph{Larsen-Pink general position} (\emph{LPgp}), so called because 
  a result of Larsen and Pink shows that broad internal subgroups of simple 
  algebraic groups are in LPgp. It is natural to ask whether 
  Proposition~\ref{p:sLPgpAb} goes through with LPgp in characteristic 0;
  in (internal) positive characteristic, pseudofinite subgroups of Lie type of 
  simple algebraic groups give a counterexample.
\end{remark}

\begin{proposition} \label{p:sLPgpAb}
  Let $G$ be a connected complex algebraic group.
  Then $G$ is commutative if and only if $\Gamma_G$ is sLPgp-coherent.
\end{proposition}
\begin{proof}
  If $G$ is commutative, then by \cite[Proposition~7.10]{BB-cohMod} $\Gamma_G$ is 
  cgp-coherent, and so in particular sLPgp-coherent.

  Suppose conversely that $\Gamma_G$ is sLPgp-coherent but $G$ is not commutative. 
  Exactly as in Lemma~\ref{l:BSGT}, we obtain a coarsely approximate subgroup $X$ 
  which contains an sLPgp internal subset $Y \subseteq  X$ with $\bdl(Y)=\bdl(X)$. In 
  particular, $\bdl(Y^n) = \bdl(Y)$ for any $n \in \N$. Rescaling, we may 
  assume $\bdl(Y)=\dim(G)$.

  We now begin to work with $\bdl$ for tuples, as in 
  \cite[\S\S2.1.5-2.1.6]{BB-cohMod} (working in the countable language 
  generated by the field structure and $Y$ by closing off under continuity). 
  Adding constants if necessary, we assume $G$ is $\emptyset $-definable.

  Now let $(f,g,h) \in Y^3$ with $\bdl(f,g,h) = 3\dim(G)$, which exists by 
  \cite[Fact~2.4]{BB-cohMod}. We proceed to calculate using additivity 
  (\cite[Fact~2.8]{BB-cohMod}) and sLPgp.

  Note first that by additivity we have $\bdl(f)=\dim(G)$, and hence by sLPgp 
  $f$ is Zariski generic over $\emptyset $ in $G$ (i.e.\ $\{f\}$ is dense in the 
  $\emptyset $-Zariski topology on $G$). Since $G$ is not abelian and any element of 
  $G$ is the product of two generics, it follows that $f$ is not central, so 
  the centraliser $C_f \leq  G$ is a proper subgroup. Also $C_f$ is infinite; it 
  is a general fact that generics of infinite algebraic groups have infinite 
  centraliser\footnote{
  For want of a clear reference, we give some proofs of this. Since we are 
  working here in characteristic 0, it suffices to see it for a complex 
  algebraic group $G$. By genericity, we can assume $f$ is in the image of 
  $\exp$, and then any 1-parameter subgroup through $f$ is an infinite 
  subgroup of $C_f$. For a characteristic-free proof, we first note that if 
  $C_f$ is finite then $f$ has finite order and hence $G$ has finite exponent. 
  Then we could cite \cite[Proposition~3.23]{poizat-groups} which shows that 
  $G$ is either nilpotent or a bad group; since algebraic groups are not bad, 
  $G$ is nilpotent. But an infinite connected nilpotent algebraic group has 
  infinite centre (\cite[Proposition~1.10]{poizat-groups}) so this contradicts 
  $C_f$ being finite. Alternatively: $G/Z(G)$ is linear (as in 
  Proposition~\ref{p:bgt}) and has finite exponent, so it cannot contain a torus, 
  and it follows that $G/Z(G)$ is unipotent, which implies that $G/Z(G)$ and 
  hence $G$ is nilpotent, which again is a contradiction.}.

  So $0 < \dim(C_f) < \dim(G)$. Since $\dim(f^G) + \dim(C_f) = \dim(G)$ (as 
  one sees by considering the conjugation action), also $0 < \dim(f^G) < 
  \dim(G)$. So sLPgp applies to $C_f$ and $f^G$ and their images under 
  definable bijections.

  Hence, as $g$ belongs to the $\{gf^h,f\}$-definable set $
    gf^h(f^G)^{-1}=\{ x : xf^G \ni gf^h \}$ we have $$\bdl(g/gf^h,f)
    < \dim(gf^h(f^G)^{-1})
    = \dim(f^G).$$
  Similarly, as $h$ belongs to the $\{f^h,f\}$-definable set $C_fh= \{ x : f^x = f^h \}$, we have $$\bdl(h/f^h,f)
    < \dim(C_fh) = \dim(C_f),$$
  so $$\bdl(f^h/f)
    = \bdl(h,f^h/f) - \bdl(h/f^h,f)
    > \dim(G) - \dim(C_f).$$

  Then
  \begin{align*} \bdl(gf^h)
    &\geq  \bdl(gf^h/f) \\
    &=  \bdl(g,gf^h/f) - \bdl(g/gf^h,f) \\
    &=  \bdl(g,f^h/f) - \bdl(g/gf^h,f) \\
    &=  \bdl(g) + \bdl(f^h/f) - \bdl(g/gf^h,f) \\
    &>  \dim(G) + \dim(G) - \dim(f^G) - \dim(C_f) \\
    &=  \dim(G). \end{align*}
  But $gf^h \in Y^4$ and $\bdl(Y^4) = \bdl(Y) = \dim(G)$. Contradiction.
\end{proof}

\begin{remark}
  The proof of commutativity in Proposition~\ref{p:sLPgpAb} did not use 
  characteristic 0, so it goes through for $G$ over an ultraproduct of 
  arbitrary fields. The converse direction is less clear.
\end{remark}

\section{Weak general position of generic arithmetic progressions}
\label{s:vectorialML}

\providecommand{\ACF}{{\operatorname{ACF}}}
\providecommand{\DCF}{{\operatorname{DCF}}}
\providecommand{\Zar}{{\operatorname{Zar}}}
\providecommand{\Stab}{{\operatorname{Stab}}}
\providecommand{\acl}{{\operatorname{acl}}}
\providecommand{\alg}{{\operatorname{alg}}}
\providecommand{\RM}{{\operatorname{RM}}}
\providecommand{\GL}{{\operatorname{GL}}}
\providecommand{\lD}{{\operatorname{lD}}}

\providecommand{\G}{{\mathbb{G}}}
\providecommand{\Q}{{\mathbb{Q}}}
\providecommand{\N}{{\mathbb{N}}}
\providecommand{\Z}{{\mathbb{Z}}}
\providecommand{\R}{{\mathbb{R}}}

\providecommand{\U}{{\mathcal{U}}}

\providecommand{\g}{{\overline{g}}}

\providecommand{\defn}[1]{{\bf #1}}

\newcommand{\ns}[1]{{#1}^*}
\newcommand{\nns}[1]{{#1}^{**}}

\subsection{Preliminaries on commutative algebraic groups}

Let $G$ be a commutative algebraic group over an algebraically closed field 
$C$ of characteristic 0.
We write $G[\infty]$ for the torsion subgroup.
Following \cite[3.1]{BP-LW}, we call $G$ \defn{almost semiabelian} if $G$ is 
connected and $G[\infty]$ is Zariski dense in $G$.

Assume $G$ is connected.
By standard results due to Chevalley and Serre, $G$ is an extension of a 
semiabelian variety by a vector group, i.e.\ we have an exact sequence over $C$ 
of algebraic groups
$$0 \rightarrow  V \rightarrow  G \xrightarrow\pi S \rightarrow  0.$$

Recall that to say that $V$ is a \emph{vector group} over $C$ means that it is 
an algebraic group over $C$ which is isomorphic over $C$ to $\G_a^n$ for some 
$n$. A homomorphism of vector groups over a field $K$ is defined via such 
isomorphisms by a matrix with entries in $K$. Since $\operatorname{char}(C) = 
0$, any algebraic subgroup $V'$ over $K \geq  C$ of $V$ is a vector subgroup 
(to see that $V'$ is closed under multiplication by scalars, note that if 
$v\in V'\setminus\{0\}$ then $\Z\cdot v\subseteq V'\cap \left< v \right>$, 
hence $V'\cap \left< v \right>= \left< v \right>$ as $\left< v \right>$ is irreducible and of 
dimension 1), and is of the form $\ker\theta$ where $\theta$ is a surjective 
homomorphism of vector groups over $K$. By linear algebra, $V'$ admits a 
\emph{vector group complement} over $K$, a vector subgroup $V'' \leq  V$ over 
$K$ with $V = V' \oplus V''$.

Recall also that a semiabelian variety is in particular a connected algebraic 
group with Zariski-dense torsion, which for each $n$ has only finitely many 
$n$-torsion points. The only other property of semiabelian varieties we will 
use is Fact~\ref{f:socle} below.

For every $n<\omega$ the $n$-torsion of $G$ (i.e.\ $\{g\in G: ng=0\}$) is 
finite. Since $C$ is algebraically closed, it follows that $G[\infty] \subseteq  
G(C)$.

\begin{lemma} \label{l:commDecomp}
  $G = G_0 \oplus V_0$ where $G_0 = G[\infty]^\Zar \leq  G$ is almost semiabelian 
  and $V_0 \leq  G$ is a vector group, both algebraic subgroups of $G$ over $C$.
\end{lemma}
\begin{proof}
  $G_0$ is over $C$ since $G[\infty] \subseteq  G(C)$.
  We have $\pi(G_0) = S$, since 
  $S[\infty]^\Zar = S$ and $\pi(G[\infty]) = S[\infty]$ as $V$ is divisible.
  Setting $V_0$ to be a vector group complement over $C$ to the vector 
  subgroup $G_0 \cap V$ within $V$,
  it follows that $G = G_0 \oplus V_0$.
  In particular, $G_0$ is connected, so $G_0$ is almost semiabelian.
\end{proof}

For $S' \leq  S$ a semiabelian subvariety,
define $$\lambda(S') := (\pi^{-1}(S')[\infty])^\Zar \leq  G.$$

\begin{lemma} \label{l:subcore}
    \begin{enumerate}[(i)]\item If $G' \leq  G$ is an algebraic subgroup then $\lambda(\pi(G')) \leq  
    G'$.
    \item $\lambda(S')$ is the unique almost semiabelian subgroup of $G$ with 
    $\pi(\lambda(S')) = S'$.
    \end{enumerate}
\end{lemma}
\begin{proof}
    \begin{enumerate}[(i)]\item Let $V'$ be a vector group complement to $G'\cap V$ within $V$.
    Then $\pi^{-1}(\pi(G'))$ is the direct sum $G' \oplus V'$ within $G$,
    so in particular $G'[\infty] = \pi^{-1}(\pi(G'))[\infty]$.
    So $\lambda(\pi(G')) = G'[\infty]^\Zar \leq  G'$.

    \item $\lambda(S')$ is almost semiabelian by Lemma~\ref{l:commDecomp} applied 
    with $\pi^{-1}(S')$ in place of $G$ (note $\pi^{-1}(S')$ is connected, as 
    it is an extension of the connected group $S'$ by the connected group 
    $V$). By divisibility of $V$, any torsion point of $S'$ lifts to a torsion 
    point of $G$, so $\pi(\lambda(S')) \supseteq S'[\infty]$, so $\pi(\lambda(S')) = 
    S'$. If $G' \leq  G$ is almost semiabelian, then (using (i)) $G' = 
    \lambda(\pi(G'))$.
    \end{enumerate}
\end{proof}

\subsection{Generic Mordell-Lang for commutative algebraic groups}

\begin{proposition} \label{p:vectorialML}
    Let $C$ be an algebraically closed field of characteristic 0.
    Let $G$ be a connected commutative algebraic group over $C$.
    Write $G = G_0 \oplus V_0$ as in Lemma~\ref{l:commDecomp}.

    Let $r > 0$, and let $\g \in G^r(K)$ be a generic point of $G^r$ over $C$ 
    in an algebraically clased extension field $K \geq  C$.
    Let $\Gamma$ be the finitely generated subgroup $\Gamma := 
    \left<{\g}\right> \leq  G(K)$.

    Let $(\ns K;\ns\Gamma) \succ  (K;\Gamma)$ be an elementary extension in the 
    language of fields expanded by a predicate for $\Gamma$.
    Let $W \subseteq  G$ be an infinite irreducible subvariety over $\ns K$, and 
    suppose $(W \cap \ns\Gamma)^\Zar = W$.

    Then $W = G_0+W'$ for some irreducible subvariety $W'$ of $V_0$.
\end{proposition}


\begin{remark}
    The option to pass to an elementary extension can be seen as a kind of 
    uniformity.
    A stronger form of uniformity would have that $((W_a \cap 
    \Gamma)^\Zar)_a$ is a constructible family if $(W_a)_a$ is.
    This holds in some related situations \cite{scanlon-automaticUniformity},
    but does not hold here.
    Consider for example the case $G=\G_a^2$ and $r=2$. Then the image of 
    $V_b = \{ x^2 + y^2 = b^2 \}$ under an element of $\GL_2(K)$ which maps 
    $\Z^2$ to $\Gamma$ has unbounded finite intersection with $\Gamma$.
\end{remark}

\begin{question}
  It would be interesting to know how far the assumptions can be weakened 
  without changing the conclusion. In particular (thanks to Udi Hrushovski for 
  suggesting this question), is it sufficient that $\g$ is contained in no 
  proper algebraic subgroup over $C$? This holds for $G$ a semiabelian 
  variety, as a consequence of the full Mordell-Lang theorem. Our proof method 
  will not handle this case.
\end{question}

\subsubsection{Preliminaries on $\DCF_0$}
To prove Proposition~\ref{p:vectorialML}, we exploit the theory of differentially 
closed fields, following \cite{Hr-ML}.
Our arguments are related to those of \cite{HP-MLTrans}, where a finer result 
along the same lines is proven for the case of a semiabelian variety.
We refer to \cite{marker-dcf} and \cite{marker-maninKernels} for general 
background on differentially closed fields and algebraic groups over them.

Work in $K = (K,\partial) \vDash  \DCF_0$. Let $C = C(K) = \ker\partial \leq  K$ be 
its constant field.
Let $G$ be a commutative algebraic group over $C$.
As explained in e.g.\ \cite[Remark~2.2]{Pillay-algDGrps},
the logarithmic derivative $\lD(g) = \lD_G(g) := d(\rho_{-g})_g(\partial g)$,
where $\rho_{-g}(h) := h-g$,
is a $C$-definable homomorphism which gives a $C$-definable exact sequence of 
commutative groups
$$0 \rightarrow  G(C) \rightarrow  G(K) \xrightarrow{\lD} LG(K) \rightarrow  0,$$
where $LG = T_0(G)$ is the tangent space at 0, which is a vector group over 
$C$.

\newcommand{\Lpi}{L\pi}
If $\pi : G \rightarrow  G'$ is a morphism of algebraic groups over $C$,
and $\Lpi : LG \rightarrow  LG'$ is the differential of $\pi$ at $0 \in G$,
then $\Lpi \circ \lD_G = \lD_{G'} \circ \pi$.

The ($K$-)Kolchin topology on $G(K)$ is a Noetherian topology refining the 
Zariski topology, with basic closed sets those defined in co-ordinates as 
zero-sets of differential polynomials. The group operations are 
Kolchin-continuous, the definable subsets are precisely the 
Kolchin-constructible subsets, and definable subgroups are Kolchin-closed.
A Kolchin-constructible set is \emph{Kolchin-irreducible} if it is not the 
union of two proper relatively Kolchin-closed subsets. 


The following fact describes the ``model-theoretic socle'' of a semiabelian 
variety over the constants; it can be seen as a direct consequence of 
\cite[Lemma~2.6]{HP-MLTrans}, or of \cite[Corollary~3.8(i)(b), 
Fact~3.3(v)]{BBP-ML-SA}.

\begin{fact} \label{f:socle}
    If $S$ is a semiabelian variety over $C = C(K)$,
    and $H \leq  S(K)$ is a connected finite Morley rank subgroup containing 
    $S(C)$,
    then $S(C)$ is the maximal connected definable subgroup of $H$ which is 
    almost internal to $C$.
\end{fact}

\subsubsection{Differential formulation of generic Mordell-Lang}

For a subset $A$ of a vector space over a field $k$, by $\left< A \right>_k$ we will denote the $k$-linear span of $A$.
\begin{proposition} \label{p:vectorialML-DCF}
    Let $G$ be a connected commutative algebraic group over $C = C(K)$.
    Write $G = G_0 \oplus V_0$ over $C$ as in Lemma~\ref{l:commDecomp}, with $G_0$ 
    almost semiabelian and $V_0$ a vector group.

    Let $0 \neq  \Gamma \leq  G(K)$ be a non-trivial abstract subgroup such that:
    \begin{enumerate}[(A)]\item No $\gamma \in \Gamma \setminus \{0\}$ is contained in any proper 
    subvariety of $G$ defined over $C$.
    \item $\left<{\lD(\Gamma)}\right>_C$ has finite dimension as a $C$-vector space.
    \item No $\xi \in \left<{\lD(\Gamma)}\right>_C \setminus \{0\}$ is contained in any proper 
    vector subgroup of $LG$ over $C$.
    \end{enumerate}

    Let $W \subseteq  G$ be an infinite irreducible subvariety over $K$, and suppose $(W \cap 
    \Gamma)^\Zar = W$.

    Then $W = G_0+W'$ for some irreducible subvariety $W'$ of $V_0$.
\end{proposition}

\begin{proof}[Proof of Proposition~\ref{p:vectorialML} from Proposition~\ref{p:vectorialML-DCF}]
    Let $C$, $G=G_0 \oplus V_0$, $\g$, $(\ns K, \ns\Gamma) \succ  (K,\Gamma)$, and 
    $W$ be as in Proposition~\ref{p:vectorialML}.

    First note, by considering automorphisms over $C$, that the truth of 
    Proposition~\ref{p:vectorialML} does not depend on the choice $\g$ of 
    realisation of the algebraic generic type over $C$ of $G^r$.
    So (possibly extending $C$) we may assume that $K \vDash  \DCF_0$,
    and $C$ is the constant field of $K$,
    and $\g \in G^r(K)$ is $\DCF$-generic over $C$,
    i.e.\ contained in no proper Kolchin-closed subset of $G^r$ defined over 
    $C$.

\newcommand{\x}{{\overline{x}}}
    Then $\Gamma$ satisfies (A-C).
    Indeed, if $n_i \in \Z$ and $\sum_i n_i g_i$ is not algebraically generic 
    in $G$ over $C$, then $\x \mapsto  \sum_i n_i x_i$ is not surjective as a map 
    $G^r \rightarrow  G$ (since it is Zariski-continuous), so $n_i = 0$ for all $i$.
    Similarly, if $c_i \in C$ and $\sum_i c_i \lD(g_i)$ is contained in a 
    proper vector subgroup over $C$,
    then $\x \mapsto  \sum_i c_i \lD(x_i)$ is not surjective as a map $G^r \rightarrow  LG$, 
    so $c_i = 0$ for all $i$.
    Finally, $\left<{\lD(\Gamma)}\right>_C = \left<{\lD(\g)}\right>_C$.

    Let $(K,\partial,\Gamma)$ be the expansion of the differential field $K$ 
    by a predicate for $\Gamma$.
    The properties (A-C) can be expressed elementarily in this language 
    (fixing the dimension in (B)).
    Let $(\nns K,\nns\partial,\nns\Gamma) \succ  (K,\partial,\Gamma)$ be an 
    elementary extension such that $(\ns K,\ns\Gamma)$ embeds elementarily in 
    the reduct $(\nns K,\nns\Gamma)$ (e.g.\ any sufficiently saturated 
    elementary extension will do).
    Then since $\ns\Gamma \cap W$ is Zariski dense in $W$, also $\nns\Gamma 
    \cap W$ is Zariski dense in $W$,
    and we conclude by applying Proposition~\ref{p:vectorialML-DCF}
    to $(\nns K,\nns\partial,\nns\Gamma)$.
\end{proof}

\begin{remark}
\newcommand{\divHull}{{\operatorname{divHull}}}
    A slightly stronger ``finite rank'' version of Proposition~\ref{p:vectorialML} 
    follows from this proof: namely, instead of taking $\Gamma := \left<{\g}\right>$, 
    we could take $\Gamma$ to be the image of any homomorphism $\theta : \Q^r 
    \rightarrow  G$ with $\theta(e_i) = g_i$.

\end{remark}

\subsubsection{Proof of Proposition~\ref{p:vectorialML-DCF}}
For a definable subset $A \subseteq  G$, let $\Stab(A) := \{ g \in G : g+A = A \}$ be 
the setwise stabiliser, which is a definable subgroup of $G$.

Since $W$ is a subvariety, $\Stab(W) \leq  G$ is an algebraic subgroup.
Quotienting, we may reduce to the case that $\Stab(W)$ contains no 
non-trivial almost semiabelian subgroup $J$.
Indeed, suppose $J \leq  G$ is such.
Then $J \leq  G_0$, and $J$ is over $C$ (as $J[\infty]\subseteq 
G[\infty]\subseteq G(C)$), and $G/J \cong  G_0/J \oplus V_0$, and $G_0/J$ is almost 
semiabelian.
Since $J$ is over $C$, property (A) of $\Gamma$ passes to $\Gamma/J$.
Writing $\pi_J : G \rightarrow  G/J$ for the quotient map,
we have $\lD_{G/J}(\Gamma/J) = \Lpi_J(\lD_G(\Gamma))$, and $\Lpi_J$ is a 
homomorphism of vector groups over $C$, so properties (B) and (C) also pass to 
$\Gamma/J$.
Furthermore, $(W/J \cap \Gamma/J)^\Zar = W/J$, and $W/J$ is an 
irreducible subvariety of $G/J$.
If $W/J$ is a point, then $W = g+J$ for some $g \in \Gamma$ (since $J$ 
stabilises $W$ and $W\cap \Gamma \neq  \emptyset $); then $J$ contains some non-trivial 
element of $\Gamma$, so since $J$ is over $C$ we must have by property (A) 
that $J=G$, and so $W = G = G_0 + V_0$.
Otherwise, $W/J$ is infinite, and applying the proposition inductively in 
$G/J$ (note $\dim(G/J)<\dim(G)$ since $J$ is infinite) we find $W/J = G_0/J + 
W''$ with $W'' \subseteq  V_0$ an irreducible subvariety;
then $W = G_0 + W''$ as required.

Recall that $G$ is an extension of a semiabelian variety $S$ by a vector 
group, all over $C$. Let $\pi : G \twoheadrightarrow  S$ be the corresponding morphism.
If $S$ is trivial, then so is $G_0$, and we conclude by setting $W' := W$.
So assume $S$ is not trivial. We argue for a contradiction.

Let $H := \lD^{-1}(\left<{\lD(\Gamma)}\right>_C) \leq  G(K)$.
So $\Gamma + G(C) \leq  H$,
and $H$ is a connected definable subgroup of finite Morley rank.
The finiteness of the Morley rank can be seen as a consequence of the 
Lascar inequalities and the fact that Morley rank coincides with Lascar 
rank for an $\omega$-stable group of finite Lascar rank \cite[Fact~1 and 
Appendix~B]{PillayPong}.

Let $X$ be a Kolchin irreducible component of $W \cap H$
such that $X \cap \Gamma$ is Zariski dense in $W$.

For $\xi \in \lD(X)$, let $X_\xi := X \cap \lD^{-1}(\xi)$.
Since $\lD^{-1}(\xi)$ is a coset of $G(C)$ and the induced structure on 
$C$ is a pure ACF, $\Stab(X_\xi) = G'_\xi(C)$ for some algebraic subgroup 
$G'_\xi \leq  G$. Then $(\pi(G'_\xi))_{\xi \in \lD(X)}$ is a definable family of 
algebraic subgroups of $S$, so by rigidity of $S$ (namely, the fact that every 
algebraic subgroup of $S$ is over $C$) and purity and stable embeddedness of 
$C$, it is a finite family.
So by irreducibility of $X$
there is some fixed semiabelian subvariety $S' \leq  S$ and a definable 
subset $\Xi \subseteq  \lD(X)$ such that $\pi(G'_\xi)^o = S'$ for $\xi \in \Xi$ and $X 
\cap \lD^{-1}(\Xi)$ is Kolchin dense in $X$, and hence Zariski dense in $W$.
Now by Lemma~\ref{l:subcore}(i), $\lambda(S') \subseteq  G'_\xi$ for $\xi \in \Xi$, so any 
element of $\lambda(S')(C)$ stabilises $X \cap \lD^{-1}(\Xi)$ and hence $W$.
By the assumption on $\Stab(W)$ (and Lemma~\ref{l:subcore}(ii)), it follows that 
$\lambda(S')$ is trivial.
By Lemma~\ref{l:subcore}(ii), we conclude that $S'$ is trivial.

\[\begin{tikzcd}
X \subseteq  G \arrow{r}{\lD} \arrow[d, shift left=1em, swap, "\pi"] &
LG \arrow[d, shift right=1em, "\Lpi"] \supseteq \Xi \\
\hphantom{X \subseteq } S \arrow{r}{\lD_S} & LS \hphantom{ \supseteq \Xi}
\end{tikzcd}
\]

\begin{claim}
    $\Lpi(\Xi)$ is a point.
\end{claim}
\begin{proof}
    We argue as in \cite[Proposition~4.3]{Hr-ML}.


    Say $\Xi$ and $X$ are defined over $d$.
    We obtain that $\lD_S^{-1}(\Lpi(\xi)) \subseteq  \acl(Cd\xi)$ for $\xi \in \Xi$.
    Indeed, let $\emptyset  \neq  U \subseteq  X_\xi$ be a set defined over $Cd\xi$ of minimal 
    Morley rank and degree. So $U$ has no non-trivial proper subset defined 
    over $Cd\xi$. Then translates of $U$ by elements of $G(C)$ coincide or are 
    disjoint, and so $U$ is a coset of its stabiliser $\Stab(U) \leq  G(C)$; 
    indeed, if $a,b \in U$ then $b-a \in G(C)$ since $U \subseteq  X_\xi \subseteq  
    \lD^{-1}(\xi)$, which is a coset of $G(C)$, and $b \in U \cap (U+(b-a))$, 
    so $U = U+(b-a)$, so $b-a \in \Stab(U)$. In particular, $\lD^{-1}(\xi)$ is 
    a union of disjoint translates of $U$ by elements of $G(C)$. Now $X_\xi$ 
    is defined over $Cd\xi$, so no translate of $X_\xi$ by an element of 
    $G(C)$ can cut $U$, so $X_\xi$ is a union of translates of $U$, so $X_\xi$ 
    is stabilised by $\Stab(U)$. So by the choice of $\Xi$ and triviality of 
    $S'$, $\pi(\Stab(U))$ and thus $\pi(U)$ are finite. So $\acl(Cd\xi) \supseteq 
    \pi(U) + S(C) = \lD_S^{-1}(\Lpi(\xi))$ (for the last equality, note that 
    $\lD_S^{-1}(\Lpi(\xi))$ is certainly a coset of $S(C)$). 

    Since $\Xi \subseteq  \lD(H) = \left<{\lD(\Gamma)}\right>_C$, it follows using property 
    (B) that $Z := \lD_S^{-1}(\Lpi(\Xi))$ is almost internal to $C$.
    Now $Z$ is a Kolchin constructible subset of the finite Morley rank 
    group $\pi(H)$, and $Z$ is Kolchin irreducible; indeed, 
    using $\Lpi \circ \lD = \lD_S \circ \pi$ we find that
    $\pi(X \cap \lD^{-1}(\Xi)) \subseteq  Z$, so $Z$ is Kolchin dense in $\pi(X) + 
    S(C)$, which is irreducible since it is a continuous image of the 
    irreducible set $X \times S(C)$.
    So by Zilber indecomposability (\cite[7.3.2]{marker-modelTheory}, noting 
    that Kolchin-irreducible sets are indecomposable) the group $\left<{Z-Z}\right>$ 
    generated by the differences is generated in finitely many steps, so is 
    also almost internal to $C$, and is connected.

    By Fact~\ref{f:socle}, we have $\left<{Z-Z}\right> = S(C)$, hence, as $\ker \lD_S=S(C)$, we have $\{0\}=\lD_S(Z-Z)=\lD_S(Z)-\lD_S(Z)=\Lpi(\Xi)-\Lpi(\Xi)$, so
     $\Lpi(\Xi)$ is a 
    point.
\end{proof}

Now $S$ is non-trivial, so $\ker\Lpi \lneq  LG$, so $\lD(H) \cap \ker\Lpi = 
\left<{\lD(\Gamma)}\right>_C \cap \ker\Lpi = \{0\}$ by property (C).
Since $\Xi \subseteq  \lD(H)$, it follows that $\Xi$ is also a point.
Then $X \cap \lD^{-1}(\Xi)$ is contained in a coset of $G(C)$ and is Kolchin 
dense in $X$, so also $X$ is contained in a coset of $G(C)$.

So letting $\alpha \in X \cap \Gamma$, we have $X-\alpha \subseteq  G(C)$.
Then $W-\alpha$ contains a Zariski-dense set of $C$-rational points, so is over $C$.
Since $W-\alpha$ is infinite,
it contains some non-trivial element of $\Gamma - \alpha = \Gamma$;
then by property (A), $W-\alpha = G$, so $W=G \geq  G_0$,
contradicting the stabiliser assumption.

\subsection{Weak general position of generic arithmetic progressions}

Let $C \leq  K$, $G = G_0 \oplus V_0$, $r \in \N$, and $\g \in G^r(K)$ be as in 
Proposition~\ref{p:vectorialML}.

For $n \in \N$, let $\Gamma_n := [-n,n]^r\cdot\g = \{\sum_i m_ig_i : m_i \in 
\Z,\; -n \leq  m_i \leq  n \} \subseteq  \Gamma := \left<{\g}\right> \leq  G(K)$.

Let $\U$ be a non-principal ultrafilter on $\N$.
Let $\nu = \lim_{i\rightarrow \U} \nu_i \in \N^{\U} \setminus \N$,
and let $\Gamma_\nu := \prod_{i\rightarrow \U} \Gamma_{\nu_i} \subseteq  \Gamma^\U \leq  G(K^\U)$;
we could also write this as $\Gamma_\nu = [-\nu,\nu]^r\cdot\g$.
Note $|\Gamma_\nu| = (2\nu+1)^r$.

\providecommand{\codim}{{\operatorname{codim}}}
\begin{theorem} \label{t:arithProgWGP}
    Assume $r \geq  \dim V_0$.
    If $W$ is a proper subvariety of $G$ over $K^\U$,
    then there is $c \in \R_{>0}$ such that
    $|W \cap \Gamma_\nu| \leq  c\max(1, \nu^{r - \codim_GW})$.
\end{theorem}
\begin{proof}
    Replacing $W$ with $(W \cap \Gamma^\U)^\Zar$,
    we may assume $W \cap \Gamma^\U$ to be Zariski dense in $W$.
    Considering each component in turn, we may assume that $W$ is irreducible.
    If $W$ is a point, we conclude by setting $c=1$.
    So we may also assume that $W$ is infinite.

    By Proposition~\ref{p:vectorialML}, $W = G_0 + W'$ where $W' \subseteq  V_0$ is 
    an irreducible subvariety of $V_0$.
    Since $W \neq  G$, $V_0$ is non-trivial.
    So by genericity of $\g$, the projection $\pi_{V_0} : G \rightarrow  V_0$ of the 
    splitting $G=G_0 \oplus V_0$ is a bijection on $\Gamma$ and hence on 
    $\Gamma^\U$.
    Also $\codim_GW = \codim_{V_0}W'$.
    So we may assume $G = V_0$.

    Let $k := \dim(V_0) = \dim(G)$.
    Then by the genericity, $g_1,\ldots ,g_k$ is a basis for $V_0(K)$ as a 
    $K$-vector space.
    Transforming $W$ by the corresponding element of $\GL^K(V_0)$, we may 
    replace $V_0$ with $\G_a^k$ and $g_i$ with $e_i$ for $1\leq i\leq k$.
    Now by a standard argument considering co-ordinate projections and 
    induction on dimension (see e.g.\ \cite[p1 or Lemma~7.1]{BB-cohMod}),
    there is $c \in \R$ such that
    $|(W - \gamma) \cap [-\nu,\nu]^k| \leq  c(\nu^{\dim W})$
    for all $\gamma \in \G_a^k$.
    Applying this for $\gamma \in [-\nu,\nu]^{r-k}\cdot(g_{k+1},\ldots ,g_r)$, we 
    obtain
    $|W \cap \Gamma_\nu| \leq  c(\nu^{\dim W})\nu^{r-k} = c(\nu^{r-\codim_GW})$.
\end{proof}

\begin{corollary}
  Taking $r := \max(1,\dim V_0)$,
  $\Gamma_\nu$ is in Larsen-Pink general position (and in particular weak 
  general position with a gap) in $G$.
\end{corollary}
\begin{proof}
  Scaling appropriately, $\bdl(\nu) = 1$.
  Then for $W$ a proper subvariety of $G$ over $K^\U$,
  $$\bdl(\Gamma_\nu \cap W) \leq  \max(0,r - \codim_GW) \leq  r\frac{\dim G - 
  \codim_GW}{\dim G}
  = \bdl(\Gamma_\nu)\frac{\dim W}{\dim G}.$$
\end{proof}

\section{Weak general position of generic nilprogressions}
\label{s:nilproWGP}

\providecommand{\Q}{\mathbb{Q}}
\providecommand{\Z}{\mathbb{Z}}

\providecommand{\U}{\mathcal{U}}

\subsection{Background on nilpotent Lie groups and algebras}

\renewcommand{\gg}{{\mathfrak{g}}}
\newcommand{\hh}{{\mathfrak{h}}}
\newcommand{\ff}{{\mathfrak{f}}}

If $G$ is a connected complex Lie group and $\gg$ is its Lie algebra,
we write $\exp_G : \gg \rightarrow  G$ (or just $\exp$) for the exponential map,
and we write $(G_k)_k$ and $(\gg_k)_k$ for the descending central series of subgroups and 
ideals: $G_0=G$, $\gg_0 = \gg$, $G_{k+1} = [G,G_k]$, $\gg_{k+1}=[\gg,\gg_k]$.

Recall that $G$ is of \emph{nilpotency class} $\leq n$ if and only if $G_n = 1$.
We also refer to the nilpotency class as the \emph{step} of a nilpotent group.

We use iterated commutator notation, both in Lie algebras and in groups:
define $[X] := X$ and recursively define
$[X_0,\ldots ,X_k] := [X_0,[X_1,\ldots ,X_k]]$ for $k \geq 1$.

Recall the Baker-Campbell-Hausdorff formula, $\exp(X)\exp(Y) = \exp(X+Y+\frac12[X,Y]+\ldots)$.
We use the following elementary consequences of this formula for nilpotent Lie algebras.

\begin{fact}\label{expSurj}
    Let $G$ be a connected complex Lie group of nilpotency class $n$ with Lie algebra $\gg$.
    \begin{enumerate}[(i)]
      \item
        For $k < n$, $\exp$ restricts to a surjection
        $\exp|_{\gg_k} : \gg_k \rightarrow  G_k$.
      \item For any $X_0,\ldots,X_{n-1} \in \gg$,
        \[ \exp([X_0,\ldots,X_{n-1}]) = [\exp(X_0),\ldots,\exp(X_{n-1})]. \]
    \end{enumerate}
\end{fact}
\begin{proof}
    By the Baker-Campbell-Hausdorff formula, we can equip $\gg$ with a group operation
    $X*Y = X+Y+\frac12[X,Y]+\ldots$ (finitely many terms) with respect to 
    which $\exp$ is a homomorphism\footnote{
    Note that $*$ is indeed associative, since $\theta : (X,Y,Z) \mapsto 
    ((X*Y)*Z)*(-(X*(Y*Z)))$ is a continuous map to the discrete set 
    $\exp^{-1}(1)$, and $\theta(0,0,0)=0$. }.
    So $\exp(\gg)$ is a subgroup, and it contains a neighbourhood of the 
    identity, so since $G$ is connected we have $\exp(\gg) = G$.
    Write $(\gg'_k)_k$ for the descending central series of subgroups of $\gg$ 
    with respect to $*$.
    Since $\exp$ is a surjective homomorphism, we have $\exp(\gg'_k) = G_k$.
    So it suffices to see that $\gg'_k = \gg_k$.

    Writing $[\cdot,\cdot]_*$ for the commutator with respect to $*$,
    and similarly for iterated commutators,
    we have
	$$[X,Y]_* = (-X)*(-Y)*X*Y = [X,Y] - \frac12[X+Y,[X,Y]] + \ldots ,$$
    so
    $$Y \in \gg_k \Rightarrow  [X,Y]_* - [X,Y] \in \gg_{k+2}.$$
    Inductively, we obtain $[X_0,\ldots ,X_k] - [X_0,\ldots ,X_k]_* \in 
    \gg_{k+1}$ for all $k \geq  0$.
    In particular, this shows (ii). We proceed to prove (i).

    It follows from the definition of $*$ (noting that $X*(-X)=0$) that each 
    $\gg_k$ is a $*$-subgroup.
    So since $\gg'_k$ is generated by elements of the form $[X_0,\ldots 
    ,X_k]_*$,
    and $[X_0,\ldots,X_k] \in \gg_k$,
    we have $\gg'_k \subseteq  \gg_k$ for all $k \geq  0$.

    We show $\gg_k \subseteq  \gg'_k$ for $0 \leq  k \leq  n$ by downwards 
    induction on $k$.
    We have $\gg_n = \{0\} \subseteq  \gg'_n$.
    As in the previous paragraph, it suffices to see that $\gg'_k$ is an 
    additive subgroup.
    If $X \in \gg'_k \subseteq  \gg_k$ and $Y \in \gg'_j \subseteq  \gg_j$ with $j\geq k$,
    then $X+Y = X*Y + Z$ with $Z \in \gg_{j+1} \subseteq  \gg'_{j+1}$ (by the 
    inductive hypothesis).
    Since $X*Y \in \gg'_k$, by induction on $j$ we obtain $X+Y \in \gg'_k$ as 
    required.
\end{proof}


\begin{fact}\label{functoriality}
Let $G$ be a connected complex Lie group with Lie algebra $\gg$.
If $H$ is a normal closed subgroup of $G$ with Lie algebra $\hh\leq \gg$, then $\gg/\hh$ can be identified with the Lie algebra of $G/H$ in such a way that the following diagram commutes, where $\pi_{G/H}:G\to G/H$ and $\pi_{\gg/\hh}$ are the corresponding quotient maps.
	\[\begin{tikzcd}
	\gg \arrow{r}{\exp_G} \arrow[swap]{d}{\pi_{\gg/\hh}} & G \arrow{d}{\pi_{G/H}} \\
	\gg/\hh \arrow{r}{\exp_{G/H}} & G/H
	\end{tikzcd}
	\]
\end{fact}
\begin{proof}
By \cite[Page 51, third paragraph]{Kn} the quotient map $\pi_{\gg/\hh}:\gg\to \gg/\hh$ is equal to the derivative of $\pi_{G/H}$. Applying \cite[Page 49, equation (1.84)]{Kn}, we get the desired result.
\end{proof}

\begin{lemma}\label{exp_inj}
	Let $G$ be a connected complex Lie group of nilpotency class $n$ with Lie algebra $\gg$. Suppose $C$ is a $\mathbb{Q}$-subalgebra of $\gg$ such that $(\exp^{-1}_{G/G_k}(e_{G/G_k}))\cap (C/\gg_k) = \{0\}$ for every $k\leq n$.
	Then $\exp_G$ is injective on $C$.
\end{lemma}

\begin{proof}

	We proceed by induction on the nilpotency class $n$ of $G$. If $n=0$ then $G$ is trivial and there is nothing to prove.
        Suppose $n>0$ and the lemma holds for groups of class $< n$. Let $m := n-1$.
	First, we verify that $C/\gg_m\leq \gg/\gg_m$ satisfies the assumptions of the lemma.
        Note that $G/G_m$ has nilpotency class $\leq m$.
        So consider any $k\leq m$ and suppose there is some non-zero $$(c/\gg_m)/(\gg/\gg_m)_k\in (C/\gg_m)/(\gg/\gg_m)_k\cap (\exp^{-1}_{(G/G_m)/(G/G_m)_k}(e)).$$

        Note that for any normal subgroup $N\trianglelefteq G$ we have $(G/N)_k=G_k/N$ and similarly $(\gg/I)_k=\gg_k/I$ for any ideal $I$ in $\gg$.

        Note also that $\gg_m$ is the Lie algebra of $G_m$; this follows for example from Fact~\ref{expSurj}(i). So Fact~\ref{functoriality} applies to $\exp : \gg/\gg_m \to G/G_m$.

	The condition $(c/\gg_m)/(\gg/\gg_m)_k\neq 0$ means that $c/\gg_k\neq 0$ since $(c/\gg_m)/(\gg/\gg_m)_k=(c/\gg_m)/(\gg_k/\gg_m)=c/\gg_k$, and $(c/\gg_m)/(\gg/\gg_m)_k\in\exp^{-1}_{(G/G_m)/(G/G_m)_k}(e)$ means, by Fact \ref{functoriality}, that $\exp(c/\gg_m)\in (G/G_m)_k=G_k/G_m$, so by Fact \ref{functoriality} again, we get $\exp(c)\in G_k$. Thus $\exp_{G/G_k}(c/\gg_k)=e_{G/G_k}$, a contradiction to the assumption about $C$.

	Now  if $x,y\in \gg$ and $\exp_G(x)=\exp_G(y)$ then $$\exp_{G/G_m}(x/\gg_m)=\exp_G(x)/G_m=\exp_G(y)/G_m=\exp_{G/G_m}(y/\gg_m),$$ so by the inductive hypothesis we have that $x/\gg_m=y/\gg_m$, hence, as $\gg_m\leq Z(\gg)$, we get that $[x,-y]=0$. So by the Baker-Campbell-Hausdorff formula $\exp_G(x-y)=\exp_G(x)\exp_G(y)^{-1}=e_G$, hence by the assumption $x-y=0$ and we are done.
\end{proof}

We will see in Remark \ref{remark-generics} below that the assumptions of the above lemma are satisfied when $C$ is generated by a generic (in a suitable sense) tuple of elements of $\gg$.

\subsection{Nilboxes}

Let $G$ be a connected complex Lie-group and $\gg$ its Lie algebra. Suppose $G$ is nilpotent. There are canonical approximate subgroups in $G$, namely \emph{nilprogressions}.  Our aim in this section is to show that ultraproducts of nilprogressions generated by sufficiently long tuples of independent generics are wgp. However, we will not work directly with nilprogressions. Instead, we work with \emph{nilboxes} that come from $\gg$, as the vector space structure of $\gg$ makes nilboxes easier to deal with. Since nilboxes and nilprogressions are closely related to each other, we can transfer statements between them.

Let us first recall the definitions of nilprogressions and nilboxes. We will follow the definitions from \cite{BG09} and \cite{BGT12}.

\begin{definition}
Let $X=\{x_0,\ldots,x_{\ell-1}\}$ be a set of elements in $G$ and $N\in\mathbb{N}$. Define the \emph{nilprogression $P(X;N)$ of rank $\ell$ and length $N$} as the set of words in $X\cup X^{-1}$ in which every $x_i$ and its inverse appear at most $N$ times between them.
\end{definition}

We will also call an ultraproduct $\prod_{i \to \U} P(X;N_i)$ 
a nilprogression of rank $|X|$ and (nonstandard) length $\lim_{i\to \U} N_i$.

Now suppose $\ff$ is a free nilpotent Lie algebra of step $n$ on generators $(x_0,\dots,x_{M-1})$.
Recall from \cite[Section 2]{BG09} that the weight vector $\chi(c)$ of a formal Lie monomial $c$ in $x_0,\dots,x_{M-1}$ counts the number of occurrences of each $x_i$ in $c$ (thus the order $\ord(c)$ of $c$ is the sum of all coordinates of $\chi(c)$).  A list of \emph{basic commutators} in $\ff$ is a sequence $(c_i)_{i<M'}$ of formal  Lie monomials of order $\leq n$ in variables $x_0,\dots,x_{M-1}$ in which monomials with the same weight vector are consecutive, monomials of the same order are consecutive, higher order monomials follow lower order monomials, and\\
1) $c_i=x_i$ for every $i<M$,\\
2) if $\ord(c_i)+\ord(c_j)\leq n$ then  $[c_i,c_j]=c_k$ for some $k>j$ if and only if $i>j\geq t$, where either $t$ is such that $c_i = [c_s,c_t]$, or $\ord(c_i)=1$ and $t=0$.

It is well known that $(c_i)_{i<M'}$ is a linear basis for $\ff$ (\cite[Chapter 11]{Hall},\cite[Section 2]{BG09}).

\begin{remark}\label{e_i_basic}
	With the above notation, one easily gets that\\ $[\dots[x_{i_2},x_{i_1}],x_{i_3}],x_{i_4}],\dots,x_{i_k}]$ is a basic commutator for any $k\leq n$ and $i_1<\dots<i_k<M$. Note that, by bilinearity and antisymmetry of the Lie bracket,  $[\dots[x_{i_2},x_{i_1}],x_{i_3}],x_{i_4}],\dots,x_{i_k}]=(-1)^{k-2}[x_{i_k},\dots,x_{i_1}]$ in $\ff$.
\end{remark}

\begin{definition}\label{def:nilbox}
  Let $X=(x_0,\dots,x_{M-1})$ be a tuple of elements in $\gg$ and $\ff(X)$ the free nilpotent $\Q$-Lie algebra of nilpotency class $n$ generated by $X$.
  Let
  $(b'_{i,j})_{i,j}$ be a list of basic commutators in $\ff(X)$ with each $b'_{i,j}$ being a Lie monomial in $x_0,\dots,x_{M-1}$ of order $i$.
  Let $b_{i,j}$ be the image of $b'_{i,j}$ in $\gg$ under the natural homomorphism which is the identity on $X$.

	Given $N\in\N$, define the \emph{nilbox $P_N(X)$ of length $N$ generated by $X$} as $$P_N(X):=\{\sum_{m\leq n}\sum_{j}t_{m,j}b_{m,j}: t_{m,j}\in[-N^{m},N^{m}]\cap\Z\}.$$
By abuse of notation, we also call the image $\exp(P_N(X))$ a nilbox in $G$.
\end{definition}

There is a similar notion of \emph{nilpotent progression} defined in \cite[Definition 1.4]{BG09}.  It is a counterpart to the notion of a nilbox. It turned out that in nilpotent groups, nilprogressions and nilpotent progressions are essentially the same in terms of mutual control. In connected nilpotent complex Lie groups, nilpotent progressions and nilboxes mutually control each other as well. Hence it is enough to deal with only one of them in our case.


\begin{definition}(\cite[Definition 1.2]{BG09})
Suppose $G$ is a group, $A,B$ are subsets of $G$ and $K>1$ is a real parameter. We say that \emph{$A$ is $K$-controlled by $B$}, or that \emph{$B$ $K$-controls $A$}, if $|B|
\leq K|A|$ and there is some set $X\subseteq G$  with $|X|\leq K$ and $A\subseteq (X\cdot B)\cap(B\cdot X)$.
\end{definition}

We summarise the situation with the following known fact about equivalence of different notions of progressions in complex nilpotent Lie-groups.
\begin{fact}\label{mut_contr}
  Suppose that $G$ is a connected Lie group over $\mathbb{C}$ of nilpotency class $s$. Let $x_1,\dots,x_\ell\in \gg$, $u_i=\exp(x_i)$ and $N>0$ an integer. Then there is a constant $C_{s,\ell}$ depending only on $s$ and $\ell$ such that the nilprogression $P(\{u_1,\dots,u_\ell\};N)$ and the nilbox $\exp(P_N(x_1,\dots,x_\ell))$ are $C_{s,\ell}$-approximate subgroups, and furthermore $P(\{u_1,\dots,u_\ell\};N)$ and $\exp(P_N(x_1,\dots,x_\ell))$ mutually $C_{s,\ell}$-control each other.
\end{fact}
\begin{proof}
The nilbox $\exp(P_N(x_1,\dots,x_\ell))$ being a $C_{s,\ell}$-approximate subgroup follows directly from \cite[Propositions 7.1 and 7.2]{BG09}. Meanwhile, the nilprogression $P(\{u_1,\dots,u_\ell\};N)$ is a $C_{s,\ell}$-approximate subgroup by \cite[Lemma 1.16]{Tao10}.

By \cite[Proposition 2.4]{BG09}, $\exp(P_N(x_1,\dots,x_\ell))$ and the nilpotent progression of length $N$ generated by $\{u_1,\ldots,u_\ell\}$ mutually $C_{s,\ell}$-control each other when $G$ is simply connected of nilpotency class $s$. By \cite[Proposition C.1]{To14}, nilpotent progressions and nilprogressions on the same set of generators of the same length mutually control each other with some constant depending only on the number of generators and the nilpotency class. Hence we get the desired result in the case that $G$ is simply connected.
For the general case, let $H$ be a simply connected Lie group with Lie algebra $\gg$. By \cite[Page 49, first paragraph]{Kn} applied to $\phi=id:\gg\to \gg$, there exists a smooth homomorphism $\Phi:H\to G$ such that $d\Phi=id$. Hence, by \cite[Page 49, equation (1.84)]{Kn}, we have that $\exp_G=\Phi \circ \exp_H$. By the simply connected case, $P(\{\exp_H(x_1),\dots,\exp_H(u_\ell)\};N)$ and $\exp_H(P_N(x_1,\dots,x_\ell))$ mutually $C_{s,\ell}$-control each other for some constant $C_{s,\ell}$. Hence, their images under $\Phi$, i.e.  $P(\{u_1,\dots,u_\ell\};N)$ and $\exp_G(P_N(x_1,\dots,x_\ell))$, mutually $C_{s,\ell}$-control each other as well. \end{proof}


\begin{lemma}\label{wgp_control}
	Suppose $G$ is a connected algebraic group.
	Let $K \in \R$.
        If $A=\prod_{i \to \U} A_i$ and  $B=\prod_{i \to \U}B_i$ are broad internal subsets of $G$ such that $A_i$ is $K$-controlled by $B_i$ for $\U$-many $i$, and $B$ is wgp [with gap $\eta$], then $A$ is wgp [with gap $\eta$].
\end{lemma}
\begin{proof}
  As $B$ $K$-controls $A$, it clearly follows that $\bdl(A)=\bdl(B)$. Now suppose $W\subsetneq G$ is Zariski-closed. Let $X$ witness the control. Then
  $$\bdl(W\cap A)\leq  \bdl (W\cap (X\cdot B))=\max_{x\in X}\bdl(W\cap x\cdot B) = \max_{x\in X}\bdl(x^{-1}W\cap B) \leq \bdl(B)-\eta < \bdl(B)$$ as each $x^{-1}W$ is a proper Zariski-closed subset of $G$.
\end{proof}


\subsection{The proof of weak general position of generic nilprogressions}

In this final subsection we are going to prove the main result  of  Section \ref{s:nilproWGP}, Theorem \ref{wgp}.

 Let $G$ be a complex algebraic group with Lie algebra $\gg$. Fixing an arbitrary choice of basis, we identify $\gg$ with $\C^{\dim\gg}$. The complex Lie algebra structure on $\gg$ is then definable in the field of complex numbers $(\mathbb{C},+,\cdot,0,1) \vDash \ACF_0$ over some finitely many parameters (the structure constants of the Lie algebra). Throughout the whole subsection, we let  $C_0 \subseteq \C$ be a countable subfield such that $\gg$ and $G$ are both definable in $(\mathbb{C},+,\cdot,0,1)$ over $C_0$.

 \begin{remark}\label{remark-generics}
 Let $C$ be a $\mathbb{Q}$-subalgebra of $\gg$ generated by independent generics in the sense of $\ACF_0$ over the countable\footnote{Here we view $\exp^{-1}_{G/G_k}(e_G/G_k)$ as a countable set of imaginaries; as $\ACF_0$ eliminates imaginaries, we can identify it with a countable subset of $\C$.} set $C_0\cup\bigcup_{k\leq n}\exp^{-1}_{G/G_k}(e_{G/G_k})$. Then $ (\exp^{-1}_{G/G_k}(e_{G/G_k}))\cap C/\gg_k=\{0\}$ for all $k\leq n$. Indeed, otherwise we obtain a $\Q$-algebra term with constant non-zero value modulo $\gg_k$, which yields a contradiction, as the value of such a term at 0 is 0. In particular, $\exp_G$ is injective on $C$ by Lemma \ref{exp_inj}.

 \end{remark}

\begin{theorem}\label{wgp}
Let $G$ be a connected complex nilpotent algebraic group of nilpotency class $n$. There are a natural number $\ell$ and $\eta>0$ such that any broad nilprogression generated by $\ell$ independent generics in $G$ is wgp with gap $\eta$.
\end{theorem}


We will reduce Theorem \ref{wgp} to the following (similar, but slightly more technical) statement about nilboxes:
\begin{theorem}\label{wgpLie}
Let $G=G(\C)$ be a connected nilpotent algebraic group over $\C$ of nilpotency class $n$. Let $\gg$ be its Lie algebra.  Then, there are a natural number $\ell=\ell(G)$ and $\eta=\eta(G)>0$ such that the following holds.

 Let $X=(x_0,\dots,x_{\ell-1})\subseteq \gg$ be such that $\exp(X)$ is a tuple of $\ell$ independent generics in $G(\C)$ over $C_0$ and $X$ is a tuple of $\ell$ independent generics in $\gg$ over $\Delta:=C_0\cup\bigcup_{k\leq n}\text{exp}^{-1}_{G/G_k}(e_G/G_k)$. Let $P:=\prod_{i\to\U}P_{N_i}(X)$ and $N^*:=\lim_{i\to\U} N_i\in \N^{\U}$ with $\bdl(N^*)>0$. Suppose $W\subseteq G(\C^{\U})$ is a proper subvariety defined over $\C^{\U}$. Then $\bdl(P\cap \exp^{-1}(W))\leq \bdl(P)-\eta$.
\end{theorem}

\begin{proof}[Proof of Theorem \ref{wgp} from Theorem \ref{wgpLie}]
Let $\ell$ and $\eta$ be given by Theorem \ref{wgpLie} applied to $G$. We claim that the same $\ell$ and $\eta$ work for Theorem \ref{wgp}.
 First we note that there exists a set $X$ as in the statement of Theorem \ref{wgpLie}. Indeed, to make sure $\exp(X)$ is generic in $G$, we need $\exp(X)\cap W_i=\emptyset$ for every $i<\omega$, where $(W_i)_{i<\omega}$ is an enumeration of all proper subvarieties of $(G(\C))^\ell$ defined over $C_0$. As each $W_i$ is nowhere dense in $(G(\C))^\ell$ and $\exp$ is a local homeomorphism, $\exp^{-1}(W_i)$ must be nowhere dense in $\gg^{\ell}$. Hence $\bigcup_i\exp^{-1}(W_i)$ is meagre in $\gg^{\ell}$. As $\Delta$ is countable, the set of non-generics over $\Delta$ in $\gg^{\ell}$ is meagre. By the Baire category theorem the set of generics in $\gg^{\ell}$ over $\Delta$ not belonging to $\bigcup_i\exp^{-1}(W_i)$ is dense, so in particular non-empty.

By Remark \ref{remark-generics}, $\exp$ is injective restricted to $P_{N_i}(X)\subseteq G(\C)$, hence also injective on $P:=\prod_{i\to\U}P_{N_i}(X)\subseteq G(\C^{\U})$. Therefore for any proper subvariety $W\subseteq G(\C^{\U})$  defined over $\C^{\U}$, we have $$\bdl(W\cap \exp(P))=\bdl( \exp^{-1}(W)\cap P)\leq \bdl(P)-\eta=\bdl(\exp(P))-\eta.$$
Hence $\exp(P)$ is wgp with gap $\eta$. As the nilprogression $P(\exp(X);N^*)$ is $C_{n,\ell}$-controlled by $\exp(P)$ by Fact \ref{mut_contr} for some constant $C_{n,\ell}$, we conclude by Lemma~\ref{wgp_control} that $P(\exp(X);N^*)$ is wgp with gap $\eta$.
Now for any set $Y$ of $\ell$ independent generics in $G$, by uniqueness of the generic type in $G^l$ we get that there is an automorphism $\sigma$ of $G(\mathbb{C})$ mapping $\exp(X)$ to $Y$. Hence $\sigma$ also maps $P(\exp(X);N^*)$ to $P(Y;N^*)$ in $G(\C^{\mathcal{U}})$. As being wgp with gap $\eta$ is invariant under automorphisms, we conclude that $P(Y; N^*)$ is wgp with gap $\eta$, as required.
\end{proof}

We will prove Theorem \ref{wgpLie} by induction on the nilpotency class $n$. The abelian case $n=1$  has been treated in Theorem \ref{t:arithProgWGP} as $\exp$ is a homomorphism and it maps nilboxes to generalised arithmetic progressions bijectively when the generator set $X$ is independent from $\exp^{-1}(e_G)$. For $n>1$, consider the following exact sequence:

$$0 \longrightarrow\gg_{n-1} \longrightarrow \gg \longrightarrow \gg/ \gg_{n-1} \longrightarrow 0.$$

Note that $\gg_{n-1}$ is abelian and $\gg/\gg_{n-1}$ is nilpotent of step $n-1$. Let $s:=\ell(G_{n-1})$,$\eta_{1}=\eta(G_{n-1})$ and $\ell':=\ell(G/G_{n-1})$, $\eta_2=\eta(G/G_{n-1})$ be numbers given by the inductive hypothesis applied to $G_{n-1}$ and $G/G_{n-1}$ respectively.
(Remember $G/G_{n-1}=\exp(\gg/\gg_{n-1})$ and $G_{n-1}=\exp(\gg_{n-1})$ by Fact \ref{functoriality} and Fact \ref{expSurj}(i).)
By Zilber's indecomposability theorem there is $d<\omega$ such that $G_{n-1}=Y^d$ where $Y := \{[g_1,\dots,g_{n}]:g_1,\dots,g_n\in G\}$. 
 Put $\eta:=\min(\eta_1,\eta_2)$ and let $X=(x_0,\dots,x_{\ell-1})$ be a tuple of $\ell:=\max\{sdn,\ell'\}$ independent generics in $\gg$ over $\Delta$ such that $\exp(X)$ is a tuple of independent generics in $G$ over $C_0$. Let $P:=\prod_{i\to\U}P_{N_i}(X)$. We want to prove $\bdl(P\cap \exp^{-1}(W))\leq \bdl(P)-\eta$ for any proper subvariety $W\subseteq G(\C^{\U})$ defined over $\C^{\U}$.

We start with some definitions.
\begin{definition}
Let $Y\subseteq \gg$ be a finite set, define $\mathcal{A}_N(Y)$, the \emph{progression} of length $N$ generated by $Y$, as $\mathcal{A}_N(Y):=\{\sum_{y\in Y}c_yy: c_y\in [-N,N]\cap \Z\}$.
\end{definition}

Let $(b'_{i,j})_{i,j}$ and $(b_{i,j})_{i,j}$ be as in Definition \ref{def:nilbox} and let $B_m:=\langle b_{m,j}:j\rangle_{\Q}$ be the $\Q$-linear span of $\{b_{m,j}:j\}$ and $B_{m,N} : = \mathcal{A}_{N^m}(\{b_{m,j}:j\})$. Note that $P_N(X)=\sum_{m\leq n}B_{m,N}$ and $B_{n,N}\subseteq \gg_{n-1}$.

\begin{claim}\label{genericsGn}
$\exp(B_{n,N})$ contains $s$ independent generics in $G_{n-1}$ (over $C_0$) for all $N>0$.
\end{claim}
\begin{proof}
  By Remark \ref{e_i_basic}, for each $k<sd$ there is $j_k$ such that\\  $(-1)^nb_{n,j_k}=[x_{kn+n-1},x_{kn+n-2},\dots,x_{kn}]$, so, by Fact \ref{expSurj}(ii),  $$\exp((-1)^nb_{n,j_k})=[\exp(x_{kn+n-1}),\exp(x_{kn+n-2}),\dots,\exp(x_{kn})].$$ Now $\exp(x_0),\dots,\exp(x_{sdn-1})$ are independent generics in $G$ by assumption, and so it follows that $\exp(b_{n,j_0}),\dots,\exp (b_{n,j_{sd-1}})$ are independent generics in $Y:=\{[g_1,\dots,g_n]:g_1,\dots,g_n\in G\}$.
  So, putting $e_i:=\sum_{id\leq k< (i+1)d}b_{n,j_k}$  and  $t_i:=\exp(e_i)$ for $i<s$, we get that $t_i=\prod_{id\leq k< (i+1)d}b_{n,j_k}$ are independent generics in $Y^d=G_{n-1}$.
  As $e_i\in B_{n,N}$, $\exp(e_i)\in \exp(B_{n,N})$, as required.
\end{proof}

In the following, we will cover $P$ by a sum of progressions of the same coarse dimension by choosing a suitable basis in $B_{m}$ for each $m$. We first develop some general facts about change of basis for progressions.

Let $t\in \N^{>0}$, and $S\subseteq \gg$. Define $S/t:=\{s/t:s\in S\}$ (as $\text{char}(\C)=0$, $s/t$ always exists.) Note that $s\mapsto s/t$ is a bijection from $S$ to $S/t$.
\begin{claim}\label{basechange}
 Let $C,D \subseteq \gg$ be finite subsets such that $\langle D\rangle_{\Q}\subseteq\langle C\rangle_{\Q}$. Then there is $t\in\N^{>0} $ such that for all $N \in \N$, $\mathcal{A}_N(D)\subseteq \mathcal{A}_{t^2|C||D|N}(C)/t$.
\end{claim}
\begin{proof}
Suppose $C=\{c_0,\ldots,c_k\}$ and $x=\sum_{b\in D}t_bb\in \mathcal{A}_N(D)$. For each $b\in D$, write $b=\sum_{i\in I_b}\frac{p_{b,i}}{q_{b,i}}c_i$ where $p_{b,i},q_{b,i}\in\Z\setminus\{0\}$ and $(p_{b,i},q_{b,i})=1$ and $I_b\subseteq \{0,\ldots,k\}$. Define $t:=\prod_{b\in D,i\in I_b}|p_{b,i}||q_{b,i}|$. 
 Then \begin{align*}
x&=\sum_{b\in D}t_bb=\sum_{b\in D}t_b\left(\sum_{i\in I_b}\frac{p_{b,i}}{q_{b,i}}c_i\right)\\
&=\sum_{i\leq k}\left(\sum_{b\in D,i\in I_b}\frac{t_bp_{b,i}}{q_{b,i}}\right)c_i.
\end{align*}
Note that $$\sum_{b\in D,i\in I_b}\frac{t_bp_{b,i}}{q_{b,i}}=\frac{1}{t}\sum_{b\in D,i\in I_b}\frac{t}{q_{b,i}}(t_bp_{b,i}),$$ with $\sum_{b\in D,i\in I_b}\frac{t}{q_{b,i}}(t_bp_{b,i})\in\Z$. Moreover, $$\left|\sum_{b\in D,i\in I_b}\frac{t}{q_{b,i}}(t_bp_{b,i})\right|\leq |C||D|N\cdot\max\{|\frac{t}{q_{b,i}}p_{b,i}|:b,i\}\leq |C||D|Nt^2.$$ Therefore, $x\in\mathcal{A}_{t^2|C||D|N}(C)/t$.
\end{proof}


\begin{claim}\label{diretsum}
Suppose $A$, $B$, and $C$ are subsets of $\gg_{n-1}(\C^{\U})$ such that $A$ and $C$ are internal, $B=A\oplus C$ (i.e. for every $b\in B$ there are unique $a\in A$ and $c\in C$ with $a+b=c$), $0<\bdl(A),\bdl(C)<\infty$ and $\exp$ injective on $B$. Suppose $\exp(C)$ is wgp with gap $\eta'$ in $G_{n-1}(\C^{\U})$, then $\exp(B)$ is also wgp with gap $\eta'$ in $G_{n-1}(\C^{\U})$. 
\end{claim}
\begin{proof}
First note that $\bdl(B)=\bdl(C)+\bdl(A)$.
Let $W$ be a proper subvariety of $G_{n-1}(\C^{\U})$. Then $$\exp^{-1}(W)\cap B=\bigsqcup_{a\in A}\exp^{-1}(W)\cap(a+C),$$
where $\bigsqcup$ denotes disjoint union. Therefore $\bdl(\exp^{-1}(W)\cap B)\leq \bdl(A)+\sup\{\bdl(\exp^{-1}(W)\cap(a+C)):a\in A\}=\bdl(A)+\sup\{\bdl((\exp^{-1}(W)-a)\cap C):a\in A\}$. Since $\gg_{n-1}$ is abelian, $\exp$ is a homomorphism, so we get $\exp^{-1}(W)-a=\exp^{-1}(W-\exp(a))$. As $\exp(C)$ is wgp with gap $\eta'$, 
we have $\bdl(\exp^{-1}(W-\exp(a))\cap C)\leq \bdl(C)-\eta'$ for all $a\in A$. We conclude that $\bdl(W\cap \exp(B))=\bdl(\exp^{-1}(W)\cap B)\leq \bdl(A)+\bdl(C)-\eta'=\bdl(B)-\eta'$.
\end{proof}

Recall that $X$ is the tuple of $\ell$ independent generics in $\gg$ and $B_m=\langle b_{m,j}:j\rangle_{\Q}$, $B_{m,N}  = \mathcal{A}_{N^m}(\{b_{m,j}:j\})$, $P_N(X)=\sum_{m\leq n}B_{m,N}$ and $P=\sum_{m\leq n}B_{m,{N^*}}$ where $N^*=\lim_{i\to\U}N_i$. Let $(e_i)_{i<s}$ be as given by Claim \ref{genericsGn} and 
put $E=\langle e_i:i<s\rangle_{\Q}$.
Consider $B_{m} \cap E$, a $\Q$-subspace of $B_{m}\cap \gg_{n-1}$. Take a basis $U_m=U_m^0\cup U_m^1\cup U_m^2$ of $B_{m}$, where $U_m^0$ is a basis of $B_{m} \cap E$ and $U_m^0\cup U_m^1$ is a basis of $B_{m}\cap \gg_{n-1}$ and such that $U_n^0$ is $\{e_i:i<s\}$ (note that $\{e_i:i<s\}$ is indeed a basis of $B_{n} \cap E=E$).  Applying Claim \ref{basechange} twice, for $m<n$ we obtain $0<t_m<t'_m$ not depending on $N$ such that
\[ B_{m,N} \subseteq \mathcal{A}_{t'_mN^m}(U_m)/t_m \subseteq \mathcal{A}_{(t'_m)^2N^m}(\{b_{m,j}:j\})/(t_m)^2. \]
Again using Claim \ref{basechange} twice, since $\langle \{b_{n,j} : j\} \cup \bigcup_{m<n} U_m^0 \rangle_\Q = \langle U_n \rangle_\Q = \langle \{b_{n,j} : j\}\rangle_\Q$, we can find $0<t_n<t_n'$ such that
$$B_{n,N}+\sum_{m<n}\mathcal{A}_{t'_mN^m}(U_m^0)/t_m \subseteq \mathcal{A}_{t'_nN^n}(U_n)/t_n
\subseteq \mathcal{A}_{(t'_n)^2N^n}(\{b_{n,j}:j\})/(t_n)^2.$$

Let $t:=\prod_{m\leq n}(t_m)^2$ and $t':=\prod_{m\leq n}(t'_m)^2$. Then for any $m\leq n$ we have  $\mathcal{A}_{(t'_m)^2N^m}(\{b_{m,j}:j\})/(t_m)^2\subseteq \mathcal{A}_{(t'N)^m}(\{b_{m,j}:j\})/t^m$ as $(t_m)^2|t^m$ and $\frac{t^m}{(t_m)^2}(t_m')^2\leq (t')^m$. Also, using multilinearity of Lie monomials,  we get $$\sum_{m\leq n}\mathcal{A}_{(t'N)^m}(\{b_{m,j}:j\})/t^m=P_{t'N}(X/t).$$

Let $$C_N:=\mathcal{A}_{t'_nN^n}(U_n)/t_n + \sum_{m<n}\mathcal{A}_{t'_mN^m}(U_m^1)/t_m \text{ and } D_N:=\sum_{m< n}\mathcal{A}_{t'_mN^m}(U_m^2)/t_m.$$
As $C_N\subseteq \gg_{n-1}$ and $\langle D_N\rangle_{\Q} \cap \gg_{n-1}=\{0\}$, we have $C_N+D_N=C_N\oplus D_N$. Then
\begin{align*} C_N \oplus D_N &\supseteq B_{n,N} + \sum_{m<n}\mathcal{A}_{t'_mN^m}(U_m^0)/t_m + \sum_{m<n}\mathcal{A}_{t'_mN^m}(U_m^1)/t_m + \sum_{m<n}\mathcal{A}_{t'_mN^m}(U_m^2)/t_m \\
&= B_{n,N} + \sum_{m<n}\mathcal{A}_{t'_mN^m}(U_m)/t_m
\supseteq \sum_{m\leq n} B_{m,N} ,\end{align*}
and
\begin{align*} C_N \oplus D_N &\subseteq \sum_{m\leq n}\mathcal{A}_{t'_mN^m}(U_m)/t_m \\
&\subseteq \sum_{m\leq n} \mathcal{A}_{(t'_m)^2N^m}(\{b_{m,j}:j\})/(t_m)^2 \\
&\subseteq \sum_{m\leq n} \mathcal{A}_{(t'N)^m}(\{b_{m,j}:j\})/t^m = P_{t'N}(X/t),\end{align*} 
so
$$P_N(X)=\sum_{m\leq n}B_{m,N}\subseteq C_N\oplus D_N\subseteq P_{t'N}(X/t).$$
Hence \begin{equation}\label{eq*}P=P_{N^*}(X)=\sum_{m\leq n}B_{m,{N^*}}\subseteq C_{N^*}\oplus D_{N^*}\subseteq P_{t'N^*}(X/t).\end{equation}

\begin{claim}\label{samefin}
$\bdl(P)=\bdl(C_{N^*})+\bdl(D_{N^*})$ and $\bdl(\pi(P))=\bdl(D_{N^*})=\bdl(\pi(P_{t'N^*}(X/t)))$ where $\pi:\gg\to \gg/\gg_{n-1}$ is the quotient map.
\end{claim}
\begin{proof}
 Note that $P_{t'N^*}(X)=\sum_{m\leq n}\mathcal{A}_{(t')^m(N^*)^m}(\{b_{m,j}:j\})$ is a sum of generalised arithmetic progressions. Thus there are finite sets $F_m$ such that $$\mathcal{A}_{(t')^m(N^*)^m}(\{b_{m,j}:j\})\subseteq F_m+\mathcal{A}_{(N^*)^m}(\{b_{m,j}:j\})$$ for all $m\leq n$. Hence,
\begin{align*}
P_{t'N^*}(X)&=\sum_{m\leq n}\mathcal{A}_{(t')^m(N^*)^m}(\{b_{m,j}:j\})\\
&\subseteq \sum_{m\leq n} F_m+\sum_{m\leq n}\mathcal{A}_{(N^*)^m}(\{b_{m,j}:j\})\\
&=\sum_{m\leq n} F_m+P_{N^*}(X).
\end{align*}
As $\sum_{m\leq n} F_m$ is a finite set, we conclude $\bdl(P_{N^*}(X))=\bdl(P_{t'N^*}(X))= \bdl(P_{t'N^*}(X/t))$, where the last equality holds since $X$ and $X/t$ have the same type over $C_0$ (as they are both generic in $\gg^{\ell}$ over $C_0$), so $P_{t'N^*}(X/t)$ is the image of $P_{t'N^*}(X)$ under an internal automorphism of $\C^\U$.

Also, by (\ref{eq*}) we have $$\bdl(P)\leq \bdl(C_{N^*})+\bdl(D_{N^*})\leq \bdl(P_{t'N^*}(X/t)),$$ which combined with the above equality gives $\bdl(P)= \bdl(C_{N^*})+\bdl(D_{N^*})=\bdl(P_{t'N^*}(X/t)).$

For the second part of the claim, note that,  as $\pi:\gg\to \gg/\gg_{n-1}$ is a Lie algebra homomorphism, we have $\pi(P_{N^*}(X))=P^{\gg/\gg_{n-1}}_{N^*}(\pi(X))$ and $\pi(P_{t'N^*}(X/t))=P^{\gg/\gg_{n-1}}_{t'N^*}(\pi(X)/t)$.
By the same arguments as above, \\ $\bdl(P^{\gg/\gg_{n-1}}_{N^*}(\pi(X)))=\bdl(P^{\gg/\gg_{n-1}}_{t'N^*}(\pi(X)))=\bdl(P^{\gg/\gg_{n-1}}_{t'N^*}(\pi(X)/t)).$ Therefore,\\ $\bdl(\pi(P_{N^*}(X)))=\bdl(\pi(P_{t'N^*}(X/t)))$. On the other hand, as $\pi(C_{N^*})=\{0\}$ and $\pi$ is injective on $D_{N^*}$, we get by (\ref{eq*}) that $$\bdl(\pi(P_{N^*}(X)))\leq \bdl(\pi(D_{N^*}))=\bdl(D_{N^*})\leq \bdl(\pi(P_{t'N^*}(X/t))),$$ so we get the desired equalities.
\end{proof}

\begin{claim}
$\exp(C_{N^*})$ is wgp with gap $\eta_1$ in $G_{n-1}(\C^{\U})$.
\end{claim}
\begin{proof}
By definition $C_{N^*}=\mathcal{A}_{t'_n(N^*)^n}(U_n^0)/t_n\oplus\sum_{m\leq n}\mathcal{A}_{t'_m(N^*)^m}(U_m^1)/t_m$. Recall that $U_n^0$ is $\{e_i:i<s\}$, where $\{\exp(e_i):i<s\}$ is a set of $s$ independent generics in $G_{n-1}$ over $C_0$ by Claim \ref{genericsGn}. Hence $\{\exp(e_i/t_n):i<s\}$ are also independent and generic in $G_{n-1}$. Let $E_{N^*}':=\mathcal{A}_{t'_n(N^*)^n}(\{e_i/t_n,i<s\})=\mathcal{A}_{t'_n(N^*)^n}(U_n^0)/t_n$. Thus, by the choice of $s$ and $\eta_1$ we know $\exp(E_{N^*}')$ is wgp with gap $\eta_1$.

As $C_{N^*}=E_{N^*}'\oplus\sum_{m\leq n}\mathcal{A}_{t'_m(N^*)^m}(U_m^1)/t_m$ and $0<\bdl(C_{N^*}),\bdl(E_{N^*}')<\infty$, we get by Claim \ref{diretsum} that $\exp(C_{N^*})$ is wgp with gap $\eta_1$.
\end{proof}

Now we are ready to prove the main theorem.\\
\textbf{Proof of Theorem \ref{wgpLie}.}

 Let $W$ be a proper subvariety of $G$ over $\C^{\U}$ and consider $\exp^{-1}(W)\cap P$; we want to show $$\bdl(\exp^{-1}(W)\cap P)\leq \bdl(P)-\eta.$$ There are two cases, the first case is that $W/G_{n-1}$ is a proper subvariety of $G/G_{n-1}$.

\begin{claim}\label{case1}
Let $W$ be a proper subvariety of $G$ over $\C^{\U}$. Suppose $W/G_{n-1}$ is a proper subvariety of $G/G_{n-1}$. Then $\bdl(\exp^{-1}(W)\cap P)\leq \bdl(P)-\eta$.
\end{claim}
\begin{proof}
  Recall that we chose $X$ to have length $|X| \geq l'$, where $l' = l(G/G_{n-1})$ is as in Theorem~\ref{wgpLie}, and $\eta_2=\eta(G/G_{n-1})$. We aim to obtain that $\pi(P_{t'N^*}(X/t))$ is wgp.
As $X/t$ is a tuple of independent generics in $\gg$ and $\pi:\gg\to \gg_{n-1}$ is definable over $C_0$, $\pi(X/t)$ is a tuple of independent generics in $\gg/\gg_{n-1}$ over $\Delta$. Moreover $\exp(\pi(X/t))$ is a tuple of independent generics in $G/G_{n-1}$ as $\exp(X)$ is a tuple of independent generics in $G$ and $\exp(\pi(x/t))=h_x$ is interalgebraic with $\exp(\pi(x))$ as $(h_x)^t=\exp(\pi(x))$ for all $x\in X$.
Since $|\pi(X/t)|\geq \ell'$ and $\pi(P_{t'N^*}(X/t))$ is the nilbox in $\gg/\gg_{n-1}$ generated by $\pi(X/t)$ of length $t'N^*$, by the inductive hypothesis (and by the choice of $\ell'$ and $\eta_2$), $\pi(P_{t'N^*}(X/t))$ is wgp with gap $\eta_2$. Therefore
\begin{align*}
 \bdl(\exp^{-1}(W)\cap P)&\leq \bdl(\exp^{-1}(W)\cap (C_{N^*}\oplus D_{N^*}))\\
 &\leq \bdl(\pi(\exp^{-1}(W)\cap (C_{N^*}\oplus D_{N^*})))+\bdl(C_{N^*})\\
 &\leq \bdl(\pi(\exp^{-1}(W)\cap P_{t'N^*}(X/t)))+\bdl(C_{N^*})\mbox{ (by (\ref{eq*}))}\\
 &\leq \bdl(\pi(\exp^{-1}(W))\cap \pi(P_{t'N^*}(X/t)))+\bdl(C_{N^*})\\
 &\leq \bdl(\pi(P_{t'N^*}(X/t)))+\bdl(C_{N^*})-\eta_2\quad(\text{by IH})\\
 &=\bdl(D_{N^*})+\bdl(C_{N^*})-\eta_2=\bdl(P)-\eta_2\quad(\text{by Claim \ref{samefin}})\\
 &\leq \bdl(P)-\eta.\qedhere
\end{align*}

\end{proof}

Now we consider the second case, i.e.\ $\dim(W/G_{n-1})=\dim(G/G_{n-1})$.
Let $$V':=\{w\in W:\dim(wG_{n-1}\cap W)=\dim(G_{n-1})\} \text{ and } V:=\{w\in W:\dim(wG_{n-1}\cap W)<\dim(G_{n-1})\}.$$
Then $\dim(G)>\dim(V')=\dim(G_{n-1})+\dim(V'/G_{n-1})$. Hence $\dim(V'/G_{n-1})<\dim(G/G_{n-1})$. Consider $(V')^{Zar}$ a proper subvariety of $G$. Then $$\dim\left((V')^{Zar}/G_{n-1}\right)\leq \dim\left((V'/G_{n-1})^{Zar}\right)=\dim(V'/G_{n-1})<\dim(G/G_{n-1}).$$
By Claim \ref{case1} applied to $(V')^{Zar}$, $\bdl(\exp^{-1}(V')\cap P)\leq \bdl(\exp^{-1}((V')^{Zar})\cap P)\leq \bdl(P)-\eta$. As $$\bdl(\exp^{-1}(W)\cap P)=\bdl(\exp^{-1}(V\cup V')\cap P)=\max\{\bdl(\exp^{-1}(V)\cap P),\bdl(\exp^{-1}(V')\cap P)\},$$ we only need to show that $\bdl(\exp^{-1}(V)\cap P)\leq \bdl(P)-\eta$. Note that for $w,w'\in W$ with $wG_{n-1}=w'G_{n-1}$ we have $w\in V\iff w'\in V$, hence $vG_{n-1}\cap V=vG_{n-1}\cap W$ is Zariski-closed for all $v\in V$.

\begin{claim}
Suppose $V$ is a constructible subset of $G$ such that $vG_{n-1}\cap V$ is Zariski-closed and $\dim(vG_{n-1}\cap V)<\dim(G_{n-1})$ for all $v\in V$. Then $\bdl(\exp^{-1}(V)\cap P)\leq \bdl(P)-\eta$.
\end{claim}
\begin{proof}
Note that, by (\ref{eq*}), $\exp^{-1}(V)\cap P\subseteq \exp^{-1}(V)\cap (C_{N^*}\oplus D_{N^*})$. Therefore, it is enough to show that $\bdl(\exp^{-1}(V)\cap (C_{N^*}\oplus D_{N^*}))\leq \bdl(C_{N^*})+\bdl(D_{N^*})-\eta=\bdl(P)-\eta$.
 As $\exp^{-1}(V)\cap (C_{N^*}\oplus D_{N^*})=\bigsqcup_{d\in D_{N^*}}\exp^{-1}(V)\cap (C_{N^*}+d)$, we get $$\bdl(\exp^{-1}(V)\cap (C_{N^*}\oplus D_{N^*}))\leq \bdl(D_{N^*})+\sup\{\bdl(\exp^{-1}(V)\cap (C_{N^*}+d)):d\in D_{N^*}\}.$$ It is enough to show that $\sup\{\bdl(\exp^{-1}(V)\cap (C_{N^*}+d)):d\in D_{N^*}\}\leq \bdl(C_{N^*})-\eta$.

We claim that \begin{equation}\label{eq7}(\exp^{-1}(V)-d)\cap \gg_{n-1}\subseteq \exp^{-1}(\exp(d)^{-1}V\cap G_{n-1}).\end{equation} Let $a-d\in \gg_{n-1}$ with $\exp(a)\in V$. Since $a-d\in \gg_{n-1}$, by the Baker-Campbell-Hausdorff formula $\exp(d)\exp(a-d)=\exp(d+a-d)=\exp(a)$. Therefore, $\exp(a-d)= \exp(d)^{-1}\exp(a)\in \exp(d)^{-1}V\cap G_{n-1}$ as desired. Now we want to show that $\dim(\exp(d)^{-1}V\cap G_{n-1})<\dim(G_{n-1})$. We may assume $\exp(d)^{-1}V\cap G_{n-1}\neq\emptyset$, then there is $g\in G_{n-1}$ such that $v_d:=\exp(d)g\in V$. Hence $$\dim(\exp(d)^{-1}V\cap G_{n-1})=\dim(V\cap \exp(d)G_{n-1})=\dim(V\cap v_dG_{n-1})<\dim(G_{n-1})$$ by our assumption on $V$.

In conclusion, we have a definable family $\{\exp(d)^{-1}V\cap G_{n-1}:d\in D_{N^*}\}$ of proper subvarieties of $G_{n-1}$. As $\exp(C_{N^*})$ is wgp with gap $\eta_1$, we have
$$\sup\{\bdl(\exp^{-1}(\exp(d)^{-1}V\cap G_{n-1})\cap C_{N^*}):d\in D_{N^*}\}\leq \bdl(C_{N^*})-\eta_1.$$ Therefore,
\begin{align*}
&\sup\{\bdl(\exp^{-1}(V)\cap (C_{N^*}+d)):d\in D_{N^*}\}\\
&=\sup\{\bdl((\exp^{-1}(V)-d) \cap C_{N^*}):d\in D_{N^*}\}\\
&=\sup\{\bdl((\exp^{-1}(V)-d)\cap \gg_{n-1}\cap C_{N^*}):d\in D_{N^*}\}\\
&\leq\sup\{\bdl(\exp^{-1}(\exp(d)^{-1}V\cap G_{n-1})\cap C_{N^*}):d\in D_{N^*}\} \mbox{ (by (\ref{eq7}))}\\
&\leq\bdl(C_{N^*})-\eta_1 \leq \bdl(C_{N^*})-\eta.\qedhere
\end{align*}

\end{proof}
This concludes the proof of Theorem \ref{wgpLie}.


\section*{Acknowledgments} 
The authors thank the anonymous referees for their careful reading and 
comments which substantially improved the paper.

\bibliographystyle{amsplain}

\begin{dajauthors}
\begin{authorinfo}[mb]
  Martin Bays\\
  Fachbereich Mathematik und Informatik\\
  Universität Münster\\
  Münster, Germany\\
  mbays@sdf.org
\end{authorinfo}
\begin{authorinfo}[jd]
  Jan Dobrowolski\\
  Department of Mathematics\\
  University of Manchester\\
  Manchester, UK\\
  dobrowol@math.uni.wroc.pl
\end{authorinfo}
\begin{authorinfo}[tz]
  Tingxiang Zou\\
  Fachbereich Mathematik und Informatik\\
  Universität Münster\\
  Münster, Germany\\
  tzou@uni-muenster.de
\end{authorinfo}
\end{dajauthors}

\end{document}